%% file: main.tex
\documentclass{amsart}

\input{StandardPaper3.tex}

\begin{document}
\title[Constrained inverse curvature flows]{Locally constrained inverse curvature flows}
\begin{abstract}
We consider inverse curvature flows in warped product manifolds, which are constrained subject to local terms of lower order, namely the radial coordinate and the generalized support function. Under various assumptions we prove longtime existence and smooth convergence to a coordinate slice. We apply this result to deduce a new Minkowski type inequality in the anti-de-Sitter Schwarzschild manifolds and a weighted isoperimetric type inequality in the hyperbolic space.
\end{abstract}
\author[J. Scheuer]{Julian Scheuer}
\address{Albert-Ludwigs-Universit\"{a}t,
Mathematisches Institut, Abteilung Reine Mathematik, Eckerstr. 1, 79104
Freiburg, Germany}
\email{julian.scheuer@math.uni-freiburg.de}
\author[C. Xia]{Chao Xia}
\address{School of Mathematical Sciences, Xiamen University, 361005,
Xiamen, P.R. China}
\email{chaoxia@xmu.edu.cn}
\date{\today}
\maketitle
\tableofcontents

\section{Introduction}

In this paper we deduce convergence results for hypersurface flows in $(n+1)$-dimensional warped product spaces 
\eq{N^{n+1}=(a,b)\x \bbS^n.}
The metric on $N$ is supposed to have the form 
\eq{\bar{g}=dr^2+\la^2(r)\s,}
where $\la$ is a positive warping factor and $\s$ is the round metric on $\bbS^n$.
Precisely, let $M=M^n$ be a closed, connected and orientable smooth manifold, then for a family of embeddings
\eq{x\cn [0,T^*)\x M\ra N,}
which satisfy the flow equation
\eq{\label{flow}\dot{x}&=\br{\fr{n}{F}-\fr{u}{\la'(r)}}\nu\\
				x(0,\cdot)&=x_0,}
we will prove long time existence and smooth convergence to a slice $\{r=\mrm{const}\}$. $F$ is a function of the principal curvatures satisfying several natural properties to be specified later,
$u$ is the {\it{support function}}
\eq{\label{supp-func}u=\bar g(\la(r)\del_r, \nu)}
and $x_0$ is an initial embedding of $M$, the image of which is a graph over $\bbS^n$,
\eq{M_0=x_0(M)=\{(r_0(y),y)\cn y\in \bbS^n\}.}

Before we state the main results in detail, cf. \cref{Sphereflow-main}, \cref{General-flow-main} and \cref{Geom-ineq}, let us give a brief overview over recent related work and our motivation to consider this flow.

Curvature driven hypersurface flows have attracted a lot of attention for about the last four decades, starting with the mean curvature flow of convex hypersurfaces, \cite{Brakke:/1978, Huisken:/1984, Huisken:/1986}, and several fully nonlinear (1-homogeneous) analogues involving the scalar curvature, the Gaussian curvature and more general functions of the principal curvatures, \cite{Andrews:/1994b, Andrews:/1994a, Chow:/1985, Chow:/1987}.
Beside these contracting flows also expanding flows for star-shaped hypersurfaces have been considered, \cite{Gerhardt:/1990,  Gerhardt:11/2011, Gerhardt:01/2014, Gerhardt:/2015, Scheuer:05/2015, Scheuer:01/2017, Urbas:/1990}. The most prominent example of an expanding flow is the inverse mean curvature flow, a weak notion of which was used by Huisken and Ilmanen to prove the Riemannian Penrose inequality, \cite{HuiskenIlmanen:/2001}. Various other applications of contracting and expanding flows include a classification of $2$-convex $n$-dimensional hypersurfaces using the mean curvature flow with surgery, due to Huisken and Sinestrari for $n\geq 3$, \cite{HuiskenSinestrari:/2009}, various extensions of geometric inequalities of Alexandrov-Fenchel-type to non-convex hypersurfaces, \cite{BrendleHungWang:01/2016}, \cite{GuanLi:08/2009}, new Alexandrov-Fenchel-type inequalities in the hyperbolic space \cite{GeWangWu:04/2014, WangXia:07/2014, WeiXiong:/2015} and in the sphere \cite{De-LimaGirao:04/2016, GiraoPinheiro:09/2016, MakowskiScheuer:11/2016, WeiXiong:/2015}.

These contracting and expanding flows all have the property of some sort of singularity formation, where however, in the optimal case, the singularities in the expanding case are quite easy to deal with and only manifest themselves in a uniform convergence to infinity or to a minimal hypersurface, if present. Still it seems tempting to directly define a flow which prevents this singularity formation, for example by adding a constraining term. The first example of such flows is the volume preserving mean curvature flow which has the form
\eq{\label{VPMCF}\dot{x}=\br{\fr{1}{|M_{t}|}\int_{M_{t}}H-H}\nu.}
It has the nice property that additionally to keeping the enclosed volume fixed it also decreases the surface area, making it a natural candidate to prove the isoperimetric inequality, once one can show that it drives hypersurfaces to round spheres. In \cite{Huisken:/1987} this was accomplished for strictly convex hypersurfaces of the Euclidean space. Similar flows, which preserve higher order curvature integrals, where considered for example in \cite{McCoy:/2003, McCoy:/2005} and in \cite{Cabezas-RivasMiquel:/2007} for flows in the hyperbolic space. Note however that the global term involved in this equation adds such heavy complications, that these nonlocal flows until know only allowed a quite restricted class of hypersurfaces, namely convex ones in the Euclidean space and horo-convex\footnote{A hypersurface in the hyperbolic space is called horo-convex if all its principal curvatures are greater or equal than $1$.} ones in the hyperbolic space. Beside some perturbation results, in the sphere there are even no results at all, \cite{AlikakosFreire:/2003}.

However, using the Minkowski identity in $\R^{n+1}$,
\eq{\int_{M}H\left<x,\nu\right>=n|M|,}
it is possible to define a constrained flow, which involves no global term and still preserves enclosed volume while decreasing the surface area. In the Euclidean space it reads
\eq{\label{GuanLi-flow}\dot{x}=(n-H\left<x,\nu\right>)\nu} 
and in warped products as above with warping factor $\la(r)$ it has to be
\eq{\label{GuanLiWang-flow}\dot{x}=(n\la'(r)-Hu)\nu,}
where $u$ is defined as \eqref{supp-func}.
This beautiful flow was invented by Guan and Li in \cite{GuanLi:/2015}, where they proved longtime existence and smooth convergence to a round sphere when the ambient space is a space form. Together with Mu-Tao Wang they generalized this result to a broader class of ambient warped products with mild assumptions on $\la$ in \cite{GuanLiWang:09/2016}. The major advantage compared to the classical volume preserving mean curvature flow \eqref{VPMCF} is that the $C^{0}$-estimates a.k.a. barriers are for free due to the maximum principle. Hence only the starshapedness of the initial hypersurface is required, namely that it is a graph in the warped product $(a,b)\x\bbS^{n}$ over the base $\bbS^{n}$. This result allows to deduce an isoperimetric inequality for such graphs in quite general warped products. See also \cite{GuanLi:/2017} for a fully nonlinear extension of this flow.

On the other hand, Brendle, Guan and Li \cite{BrendleGuanLi:/} designed an inverse type constrained curvature flow in space forms,
\eq{\label{GuanLi3}
\dot{x}=\left(\frac{n\la'}{F}-u\right)\nu.
}
Compared to the mean curvature type constrained flow \eqref{GuanLiWang-flow}, this flow seems more appropriate for higher order isoperimetric type inequalities -- the Alexandrov-Fenchel type inequalities for quermassintegrals -- in space forms, for the reason that the higher order Minkowski identities imply that for 
\eq{F=n\fr{H_{k}}{H_{k-1}}}
the $k$-th quermassintegral is preserved, while the $(k+1)$-th quermassintegral is decreasing. However, the study of \eqref{GuanLi3} is quite subtle from the PDE point of view and until today no satisfactory complete result has been achieved. Some convergence results are proved in \cite{BrendleGuanLi:/} when the initial hypersurface is already close to a sphere. A full convergence result for closed, starshaped and $k$-convex initial hypersurfaces would prove the quermass Alexandrov-Fenchel inequalities for such hypersurfaces.
For horo-convex domains these have been established by Wang and the second author \cite{WangXia:07/2014} using a global quermassintegral preserving curvature flow.

Guan-Li's considerations motivate us to study another kind of constrained flow, the constrained inverse curvature flow \eqref{flow} in general warped product spaces. Compared to \eqref{GuanLi3}, we are able to  prove the longtime existence and smooth convergence of \eqref{flow} to a coordinate slice under mild assumptions on the curvature function $F$, the warping factor $\la$ and the initial hypersurface.
We use this result to deduce a new geometric inequality in the anti-de-Sitter Schwarzschild manifolds, cf. \cref{Geom-ineq}, on which we will give more comments later.

Let us first state the main results of this paper. Since our assumptions on the curvature function and the initial embedding  depend on the structure of the warping factor $\la$, we split our flow results into two theorems. We start with the ambient space $N=\bbS_+^{n+1}$, in which case $\la(r)=\sin r, r\in [0,\frac{\pi}{2})$.

\begin{thm}\label{Sphereflow-main}
Let $x_0(M)$ be the embedding of a closed $n$-dimensional manifold $M$ into $\bbS^{n+1}$, such that $x_0(M)$ is strictly convex. 
Let 
\eq{F=n\fr{H_k}{H_{k-1}},}
where $H_{k}$ is the $k$-th normalized elementary symmetric polynomial of the principal curvatures.
Then any solution $x$ of \eqref{flow} exists for all positive times and converges to a geodesic slice in the $C^{\8}$-topology. 
\end{thm}

Now we come to ambient spaces satisfying $\la''\geq 0$. We obtain convergence results for a large class of speeds and therefore make the following assumption. 

\begin{assum}\label{F}
Let $\G\sub\R^{n}$ be a symmetric, convex, open cone containing 
\eq{\G_+=\{(\ka_i)\in \R^n\cn \ka_i>0\}}
and suppose that $F$ is positive in $\G$, strictly monotone, homogeneous of degree one and concave with
\eq{F_{|\del\G}=0,\q F(1,\dots,1)=n.} 
\end{assum} 

\begin{thm}\label{General-flow-main}
Let $a,b\in\R$ and $(N,\bar g)$ be the warped space $((a,b)\x \bbS^n,dr^2+\la^2(r)\s)$ with $\la>0$, $\la'>0$ and $\la''\geq 0$.
Let $F\in C^{\8}(\G)$ satisfy \cref{F} and let $x_0(M)$ be the embedding of a closed $n$-dimensional manifold $M$ into $N$, such that $x_0(M)$ is a graph over the domain $\bbS^n$ and such that $\ka\in \G$ for all $n$-tuples of principal curvatures along $x_0(M)$. Then any solution $x$ of \eqref{flow} exists for all positive times and converges to a geodesic slice in the $C^{\8}$-topology.
 \end{thm}
 
\begin{rem}The assumption $\la''\geq 0$ is only used for deriving the uniform lower bound for $F$. This assumption resembles the non-positivity of the ambient sectional curvature in the radial direction, a property which was also crucial in the deduction of long-time existence of the inverse mean curvature flow in warped product spaces, cf. \cite{Scheuer:01/2017}.
 \end{rem}
 Note that compared to the purely expanding inverse mean curvature flow
 \eq{\label{ICF}\dot{x}=\fr{1}{H}\nu,}
 which was treated in general warped products in \cite{Scheuer:01/2017}, the set of assumptions on the warping factor in \cref{General-flow-main} is quite small. In order to obtain convergence results of a purely expanding flow, ones needs a lot of more global information about the ambient space. From the viewpoint of geometric inequalities for hypersurfaces, only local information is required and hence a constrained flow seems to be more promising than a flow of the form \eqref{ICF}. Indeed, in this paper, we use \cref{General-flow-main} to obtain the following geometric inequalities, one weighted Min\-kowski-type inequality and one weighted isoperimetric type inequality.
 
 \begin{thm}\label{Geom-ineq}
 Let $N=(a,b)\x \bbS^{n}$ be equipped with one of the anti-de-Sitter Schwarzschild metrics or the hyperbolic metric, i.e.
 \eq{\la'=\rt{1+\la^{2}-m\la^{1-n}	}, \q m\geq 0.}
Let $\Si\sub N$ be a closed, star-shaped and mean-convex hypersurface, given by the function $r\cn \bbS^{n}\ra (a,b)$, and let \eq{\Om=\{(s,y)\in N\cn a\leq s\leq r(y),\q y\in \bbS^{n} \}.} Then there hold
\eq{\label{Geom-ineq-1}\int_{\Si}H\la'd\mu-2n\int_{\Om}\fr{\la'\la''}{\la}dN\geq \xi_{1}(|\Si|)}
and 
\eq{\label{Geom-ineq-2}\int_{\Si}H\la'd\mu-2n\int_{\Om}\fr{\la'\la''}{\la}dN\geq \xi_{0}\br{\int_{\Om}\la'dN},}
where $\xi_{0}$, $\xi_{1}$ are the associated monotonically increasing functions for radial coordinate slices. Equality holds if and only if $\Si$ is a radial coordinate slice. 
 \end{thm}
 In particular, in the hyperbolic space, due to $\la''=\la$, inequality \eqref{Geom-ineq-2} reduces to
 \eq{\label{ineqnew}
\int_{\Sigma} H\la'd\mu-(n+1)n\int_{\Omega}\la' dN \ge n|\mathbb{S}^n|^{\frac{2}{n+1}}\left((n+1)\int_{\Omega} \la' dN\right)^{\frac{n-1}{n+1}},}
where $\la'(r)=\cosh r$. Equality in \eqref{ineqnew} holds if and only if $\Si$ is a geodesic sphere centered at the origin.
 The second author proved a Minkowski type inequality in \cite{Xia:/2016a} stating that for a closed horo-convex hypersurface $\Si\sub\mathbb{H}^{n+1}$ there holds
\eq{
\left(\int_\Sigma \la'd\mu\right)^2\ge \frac{n+1}{n} \int_{\Sigma} H\la'd\mu \int_{\Omega} \la' dN.
}

Combining this with \eqref{ineqnew}, we get:
\begin{thm}\label{Geom-Ineq-B}
Let $\Sigma$ be a closed horo-convex hypersurface in $\mathbb{H}^{n+1}$ with the origin lying inside $\Omega$. Then
\begin{eqnarray*}
&&\int_\Sigma \la'd\mu\ge \left[\left((n+1)\int_\Omega \la'dN\right)^2+|\mathbb{S}^n|^{\frac{2}{n+1}}\left((n+1)\int_\Omega \la' dN\right)^{\frac{2n}{n+1}}\right]^{\frac12}.
\end{eqnarray*}
 Equality holds if and only if $\Sigma$ is a geodesic sphere centered at the origin.
\end{thm}

\begin{rem}
\cref{Geom-Ineq-B} already appeared in the paper \cite{GeWangWu:10/2015}, where it is the case $k=0$ in Thm.~9.2. However, their proof relies on an {\it{invalid}} inequality, namely \cite[equ.~(9.8)]{GeWangWu:10/2015}, which states
\eq{|\Si|^{\fr{n+1}{n}}\geq |\mathbb{S}^n|^{\frac1n}\int_{\Si}u d\mu \quad\left(=|\mathbb{S}^n|^{\frac1n}\int_{\Omega}\la' dN\right).}
This inequality is already incorrect on geodesic spheres not centered at the origin. \cref{Geom-Ineq-B} fixes this gap in the proof of \cite[Thm.~9.2]{GeWangWu:10/2015}.

\end{rem}

\begin{rem}  By using the classical inverse mean curvature flow, Brendle-Hung-Wang proved in \cite{BrendleHungWang:01/2016} for a closed, star-shaped and mean-convex hypersurface $\Sigma$ in anti-de-Sitter Schwarzschild space, that
  \eq{\label{BHW-1}
\int_{\Sigma} H\la'd\mu-(n+1)n\int_{\Omega}\la' dN \ge n|\mathbb{S}^n|^{\frac1n}\left( |\Sigma|^{\frac{n-1}{n}}- |\del N|^{\frac{n-1}{n}}\right).
}
In particular, in the hyperbolic space, they get 
  \eq{\label{BHW-2}
\int_{\Sigma} H\la'd\mu-(n+1)n\int_{\Omega}\la' dN \ge n|\mathbb{S}^n|^{\frac1n} |\Sigma|^{\frac{n-1}{n}}.}
\eqref{BHW-1} is different from \eqref{Geom-ineq-1}, in the sense that the right hand side of \eqref{Geom-ineq-1} does not depend on the horizon $\{a\}\x \bbS^{n}$.   
\end{rem}

Another nice corollary is given by the following area bound for star-shaped and mean convex hypersurfaces in ambient spaces of non-positive radial curvature. It is neither clear to the authors, whether this bound is evident by other means, nor if it has been recorded before. It follows from \cref{General-flow-main}, the monotonicity of area in these spaces, cf. \eqref{Mon-area}, and \cref{Barriers}.

\begin{cor}
Let $a,b\in\R$ and $(N,\bar g)$ be the warped space $((a,b)\x \bbS^n,dr^2+\la^2(r)\s)$ with $\la>0$, $\la'>0$ and $\la''\geq 0$. Let $\Sigma\sub N$ be a closed, star-shaped and mean-convex hypersurface,
\eq{\Sigma=\{(r(y),y)\in N\cn y\in \bbS^n \}.}
Then the area of $\Sigma$ satisfies
\eq{|\Sigma|\leq |\mathbb{S}^n|\la\br{r_{max}}^n,}
where $r_{max}=\max_{\bbS^n} r.$
\end{cor}

It would be very interesting to find further monotone quantities along these flows, in particular in a spherical ambient space.

The paper is organized as follows. In \cref{Notation,Ev}, we collect the notation and derive the fundamental evolution equations for several geometric quantities. In \cref{sec:Bound-F,sec:Grad-Bound,sec:Pres-Conv,sec:Speed-Bound}, we derive a priori estimates under various conditions on $F$ and $\la$ and in \cref{sec:Pf-Main} we complete the proof of \cref{Sphereflow-main} and \cref{General-flow-main}. \Cref{sec:Geom-Ineq} is devoted to prove monotonicity for various geometric quantities and in turn the geometric inequalities in \cref{Geom-ineq}.

\medskip

\section{Notation and conventions}\label{Notation}
\subsection{Conventions on Riemannian geometry}
\subsubsection*{Intrinsic Curvature}
Let $(M^n,g)$ be a Riemannian manifold. With respect to a local frame $(e_i)_{1\leq i\leq n}$ of the tangent bundle, let $(g_{ij})$ denote the coordinate functions of $g$ with respect to the basis $(\e^i\otimes \e^j)_{1\leq i,j\leq n}$, where $\e^i$ denote the basis elements dual to $e_i$. Let $(g^{ij})$ denote the inverse matrix of $(g_{ij})$. For a $(k,l)$-tensor field $T$, the coordinates of which with respect to this frame are given by
\eq{T=(T^{i_1\dots i_k}_{j_1\dots j_l}),}
we can define $(k+1,l-1)$-tensor fields by using the tangent-cotangent isomorphism induced by $g$, e.g.
\eq{T^{i_1\dots i_{k+1}}_{j_1\dots j_{l-1}}=T^{i_1\dots i_k}_{j_1\dots j_l}g^{j_l i_{k+1}}.}
Of course we can also raise other indices to different slots, but it will always be apparent, or explicitly stated, which
one is meant.

The Lie-Bracket of two vector fields $X,Y$ on $M$ is given by
\eq{[X,Y]\p=X(Y\p)-Y(X\p)\q\fa \p\in C^{\8}(M).}

Let $\n$ be the Levi-Civita connection of $g$, then for a $(k,l)$ tensor field $T$, its covariant derivative $\n T$ is a $(k,l+1)$ tensor field given by
\eq{&(\n T)(Y^1,\dots, Y^k,X_1,\dots,X_l,X)\\
=~&(\n_{X} T)(Y^1,\dots, Y^k,X_1,\dots,X_{l})\\
						=~&X(T(Y^1,\dots,Y^k,X_1,\dots,X_l))-T(\n_X Y^1,Y^2,\dots, Y^k,X_1,\dots,X_l)-\ldots\\
                        \hp{=}~&-T(Y^1,\dots,Y^k,X_1,\dots,X_{l-1}\n_X X_l).}
We denote by $\n^m T$ the $m$-th covariant derivative of $T$ and its coordinates with respect to a basis $(e_i)_{1\leq i\leq n}$ are denoted by
\eq{\n^m T=\br{T^{i_1\dots i_k}_{j_1\dots j_l;j_{l+1}\dots j_{l+m}}},}
where all indices appearing after the semicolon indicate covariant derivatives. The $(1,3)$ Riemannian curvature tensor is defined by
\eq{\label{Rm}\Rm(X,Y)Z=\n_X\n_Y Z-\n_Y\n_X Z-\n_{[X,Y]}Z,}
or with respect to the basis $(e_i),$
\eq{\Rm(e_i,e_j)e_k={R_{ijk}}^l e_l,}
where we use the summation convention (and will henceforth do so). The coordinate expression of \eqref{Rm}, the so-called {\it{Ricci-identities}}, read
\eq{\label{Ricci}X^k_{\ ;ij}-X^k_{\ ;ji}=-{R_{ijm}}^k X^m}
for all vector fields $X=(X^k)$.
We also denote the $(0,4)$ version of the curvature tensor by $\Rm$,
\eq{\Rm(W,X,Y,Z)=g(\Rm(W,X)Y,Z).}
The Ricci curvature can unambiguously defined in coordinates by
\eq{\Rc(e_i,e_j)=R_{ij}={R_{kij}}^k.}
 The scalar curvature is
\eq{R=R_i^{\ i}=g^{ki}R_{ki}.}

\subsubsection*{Extrinsic curvature}
When dealing with immersed hypersurfaces
\eq{x\cn M\hra N}
of a Riemannian manifold $M^n$ into an ambient Riemannian manifold $N^{n+1}$, we furnish all the previous geometric quantities of $N$ with an overbar, e.g. $\bar g=(\bar{g}_{\al\be})$, where greek indices run from $0$ to $n$, $\bar\n$ etc. We keep using latin indices, running from $1$ to $n$, for geometric quantities of $M$, e.g. the induced metric $g=(g_{ij})$. The induced geometry of $M$ is governed by the following relations. The (local) second fundamental form $h=(h_{ij})$ is given by the Gaussian formula
\eq{\label{GF}\bar\n_{X}Y=\n_X Y-h(X,Y)\nu,}
where $\nu$ is a local normal field. Note that here (and for the rest of the paper), we will abuse notation by disregarding the necessity to distinguish between a vector $X\in T_p M$ and its push-forward $x_{\ast}X\in T_p N$. The Weingarten endomorphism $A=(h^i_j)$ is given by $h^i_j=g^{ki}h_{kj}$ and there holds the Weingarten equation
\eq{\label{Weingarten} \bar\n_X \nu=A(X), }
or in coordinates
\eq{\nu^{\al}_{\ ;i}=h^k_i x^{\al}_{\ ;k}.}
We also have the Codazzi equation
\eq{\label{Codazzi}\n_Z h(X,Y)-\n_Y h(X,Z)=-\ov{\Rm}(\nu,X,Y,Z), }
or
\eq{h_{ij;k}-h_{ik;j}=-\bar{R}_{\al\be\g\de}\nu^{\al}x^{\be}_{\ ;i}x^{\g}_{\ ;j}x^{\de}_{\ ;k},}
and the Gauss equation
\eq{\label{GE}\Rm(W,X,Y,Z)=\ov{\Rm}(W,X,Y,Z)+h(W,Z)h(X,Y)-h(W,Y)h(X,Z)}
or
\eq{R_{ijkl}=\bar{R}_{\al\be\g\de}x^{\al}_{\ ;i}x^{\be}_{\ ;j}x^{\g}_{\ ;k}x^{\de}_{\ ;l}+h_{il}h_{jk}-h_{ik}h_{jl}.}
\subsubsection*{Graphs in warped products}
In this paper we deal with warped products
\eq{N=(a,b)\x \bbS^n}
with metric
\eq{\bar{g}=dr^2+\la^2(r)\s,}
where $\s$ is the round metric of $\bbS^n$. We need the specific structure of the Ricci curvature tensor in such a warped product. There holds
\eq{\label{Ricci-WP}
\ov\Rc&=-\left(\frac{\lambda''}{\lambda}-(n-1)\frac{1-\lambda'^2}{\lambda^2}\right)\bar g-(n-1)\left(\frac{\lambda''}{\lambda}+\frac{1-\lambda'^2}{\lambda^2}\right)dr\otimes dr,
}
cf. \cite[Prop. 2.1]{Brendle:06/2013}.

 Our hypersurfaces
\eq{x\cn M\hra N}
will all be graphs over $\bbS^n,$
\eq{x(M)=\{(r(y),y)\cn y\in \bbS^n\}=\{(r(y(\xi)),y(\xi))\cn \xi \in M\},}
where we do not make a notational difference between the radial coordinate $r$ of $N$ and the function $r_{|M}$. Along $M$ we will always pick the {\it{outward}} pointing normal 
\eq{\nu=v^{-1}(1,-\la^{-2}\s^{ik}\del_k r),}
where 
\eq{v^2=1+\la^{-2}\s^{ij}\del_i r\del_j r,}
and use this normal in the Gaussian formula \eqref{GF}. The support function of $M$ is defined by
\eq{u=\bar{g}(\la\del_r,\nu)=\fr{\la}{v}.}
There is also a relation between the second fundamental form and the radial function on the hypersurface. Let 
\eq{\bar h=\la'\la\s,}
then there holds
\eq{\label{graph-h2}v^{-1}h=-\n^2 r+\bar h,}
cf. 
\cite[equ.~(1.5.10)]{Gerhardt:/2006}. Since the induced metric is given by
\eq{g_{ij}=r_{;i}r_{;j}+\la^2\s_{ij},}
we obtain
\eq{\label{graph-h}v^{-1}h_{ij}=-r_{;ij}+\fr{\la'}{\la}g_{ij}-\fr{\la'}{\la}r_{;i}r_{;j}.}

Define \eq{\varphi(r)=\int_{a}^r \frac{1}{\lambda(r)}.} Regarding $r$ as a function on $\mathbb{S}^n$, we have
\eq{h_j^i=\frac{\la'}{\la v}\de_i^j-\frac{1}{\la v}\tilde{g}^{jk}\varphi_{,ki},}
where \eq{\tilde{g}^{ij}=\sigma^{ij}-\frac{\varphi_{,}^{\ i}\varphi_{,}^{\ j}}{v^2}} and the covariant derivative and index raising is performed with respect to the spherical metric $\sigma_{ij}$, cf. \cite[equ.~(3.26)]{Gerhardt:11/2011}. We will use $\hat \n$ to denote the covariant derivative on $\mathbb{S}^n$ throughout this paper. 

\subsubsection*{Anti-de-Sitter Schwarzschild space}

The anti-de-Sitter Schwarzschild manifolds are asymptotically hyperbolic Riemannian warped products of the form
\eq{N=(r_{0},\8)\x \bbS^{n}}
equipped with the warped product metric
\eq{\bar{g}=dr^{2}+\la^{2}(r)\s,}
where $\la$ satisfies
\eq{\la'=\rt{1+\la^{2}-m\la^{1-n}}}
with $m>0$ and horizon $\del N=\{r_{0}\}\x\bbS^{n}$. The limiting case $m= 0$ is the hyperbolic metric. These Riemannian manifolds carry the property to be {\it{static}}, i.e.
\eq{\bar\D \la' \bar g-\bar\n^{2}\la'+\la'\ov\Rc=0,}
which ensures that the Lorentzian warped product $-\la'^{2}dt^{2}+\bar{g}$
is a solution to Einstein's equation.

\subsection{Curvature functions}
In \cref{F}, the part of our normal variation that depends on the curvature of the hypersurface, was stipulated to depend on the principal curvatures
\eq{F=F(\ka_i).}
However, in the calculation of the evolution equations it is often useful to consider $F$ as a function of the diagonalizable Weingarten operator $A$,
\eq{F=F(A):=F(\mrm{EV}(A)),}
where $\mrm{EV}(A)$ is the unordered $n$-tuple of eigenvalues of $A$. This is well-defined due to the symmetry of $F$. However, when using this definition, $F$ is not defined on the whole endomorphism bundle, but only on the diagonalizable operators. It is thus most convenient to consider the function defined by,
\eq{\hat{F}(g,h):=F\br{\fr 12 g^{ik}(h_{kj}+h_{jk})}}
for all positive definite $g$ and all bilinear forms $h\in T^{0,2}_p M$.
Then 
\eq{\hat F^{ij}=\fr{\del F}{\del h_{ij}}}
is a $(2,0)$-tensor and we also write
\eq{\hat F^{ij,kl}=\fr{\del F}{\del h_{ij}\del h_{kl}}.}
Furthermore, if $F=F(\ka_i)$ is strictly monotone, then $\hat{F}^{ij}$ is strictly elliptic. If $F$ is concave, then
\eq{\hat F^{ij,kl}\eta_{ij}\eta_{kl}\leq 0}
for all symmetric $(\eta_{ij})$. We refer to \cite{Andrews:/2007}, \cite[Ch.~2]{Gerhardt:/2006} and \cite{Scheuer:03/2017} for more details on curvature functions. 

Furthermore we will abuse notation and also write $F$ for $\hat{F},$ since no confusion will be possible. E.g., when writing $F^{ij}$, we can only mean $\hat{F}^{ij}$, since there are two contravariant indices. 

Let us denote by $\s_k$ the $k$-th elementary symmetric polynomial and define the $k$-th normalized elementary symmetric polynomial by
\eq{H_k=\fr{1}{\binom{n}{k}}\s_k.}
Denote by $\Gamma_k$ the connected component of $\{\s_k>0\}$ which contains the point $(1,\dots,1)$.

\medskip

\section{Evolution equations}\label{Ev}

In this section we deduce the evolution equations relevant to study the flow
\eq{\label{Flow}\dot{x}=\br{\fr{n}{F}-\fr{u}{\la'}}\nu\equiv \cF\nu.}
The following basic evolution equations are well known and can be found in many places. We use the reference \cite[Ch.~2.3]{Gerhardt:/2006}, where we note that we use the other sign on the curvature tensor.

\begin{lemma}
Along \eqref{Flow} the following evolution equations hold:
\eq{\dot{g}=2\cF h,}
\eq{\fr{\bar\n}{dt}\nu=-\grad\cF,}
\eq{\label{EV-W}\dot{h}^j_i=-\cF^{\hp{;i} j}_{;i}-\cF h^j_kh^k_i-\cF\bar{R}_{\al\be\g\de}x^{\al}_{\ ;i}\nu^{\be}\nu^{\g}x^{\de}_{\ ;k}g^{kj},}
and
\eq{\label{Ev-h}\dot{h}_{ij}=-\cF_{;ij}+\cF h_{ik}h^k_j-\cF\bar{R}_{\al\be\g\de}x^{\al}_{\ ;i}\nu^{\be}\nu^{\g}x^{\de}_{\ ;j}.}
\end{lemma}

We need some further special evolution equations.

\begin{lemma}\label{Ev-Eq}
Define the operator $\cL$ by
\eq{\cL=\del_t-\fr{n}{F^2}F^{ij}\n^2_{ij}-\fr{\la}{\la'}{r_{;}}^k \n_k.}
Along the flow \eqref{Flow} of graphs 
\eq{M_t=\{(r(t,y),y)\cn y\in \bbS^n\}} we have the following evolution equations for the radial function $r$, the support function $u$ and the curvature function $F$:

\eq{\label{Ev-r}\cL r=\fr{2n}{vF}-\fr{\la}{\la'}-\fr{n\la'}{\la F^2}F^{ij}g_{ij}+\fr{n\la'}{\la F^2}F^{ij}r_{;i}r_{;j},}

\eq{\label{Ev-u}\cL u&=\fr{n}{F^2}\br{F^{ij}h_{ik}h^k_j-\fr 1n F^2}u-\fr{\la''\la}{\la'^2}\|\n r\|^2u\\
			&\hp{=}+\fr{n\la}{F^2}F^{ij}\bar{R}_{\al\be\g\de}\nu^{\al}x^{\be}_{\ ;i}x^{\g}_{\ ;m}x^{\de}_{\ ;j}{r_{;}}^m,} 

\eq{\label{EV-F}\cL F&=-\frac{2n}{F^3}F^{ij}F_{;i}F_{;j}-\frac{n}{F}\left(F^{ij}h_{jk}h^k_i-\frac1n F^2\right)+\frac{u^2\la''}{\la\la'^2}F-\frac{u\la''}{\la\la'}F^{ij}g_{ij}\\
&\hp{=}+\frac{u}{\la'^2}\left(\frac{\la'\la''}{\la}-\la'''+\frac{2{\la''}^2}{\la'}\right)F^{ij}r_{;i}r_{;j}-\frac{2\la''}{\la'^2}F^{ij}u_{;i} r_{;j}\\
&\hp{=}-\frac{\la}{\la'}F^{ij}\bar{R}_{\al\be\g\de}\nu^{\al}x^{\be}_{\ ;i}x^{\g}_{\ ;m}x^{\de}_{\ ;j}{r_{;}}^m-\left(\frac{n}{F}-\frac{u}{\la'}\right)F^{ij}\bar{R}_{\al\be\g\de}x^{\al}_{\ ;i}\nu^{\be}\nu^{\g}x^{\de}_{\ ;j}.}
\end{lemma}

\pf{
The $0$-component of \eqref{Flow} gives
\eq{\dot{r}=\cF v^{-1}=\br{\fr{n}{F}-\fr{u}{\la'}}v^{-1},}
while from \eqref{graph-h} we see, using the $1$-homogeneity of $F$,
\eq{-\fr{n}{F^2}F^{ij}r_{;ij}=\fr{n}{vF}-\fr{n\la'}{\la F^2}F^{ij}g_{ij}+\fr{n\la'}{\la F^2}F^{ij}r_{;i}r_{;j}.}
Adding up gives \eqref{Ev-r}.

To prove \eqref{Ev-u}, note that $\la\del_r$ is a conformal vector field, i.e. for all ambient vector fields $\bar X$ there holds
\eq{\bar\n_{\bar X}(\la\del_r)=\la' \bar X.}
Hence
\eq{\label{Ev-u-1}\dot{u}=\bar{g}(\la'\dot{x},\nu)+\bar{g}(\la\del_r,\bar\n_{\dot{x}}\nu)=\la'\cF-\bar{g}(\la\del_r,\grad\cF).}
Furthermore there holds
\eq{Xu=\bar{g}(\la\del_r,A(X))}
and
\eq{\n^2u(X,Y)&=Y(Xu)-(\n_Y X)u\\
			&=\la'h(X,Y)+\bar{g}(\la\del_r,\n_Y A(X))-h(Y,A(X))u\q \fa X, Y\in TM.}
We use the Codazzi equation \eqref{Codazzi} to deduce
\eq{\bar{g}(\la\del_r,\n_YA(X))&=\la\bar{g}_{\al\be}r^{\al}x^{\be}_{\ ;k}h^k_{i;j}X^iY^j\\
				&=\la \bar{g}_{\al\be}r^{\al}x^{\be}_{\ ;k}{h_{ij;}}^k X^i Y^j-\la\bar{g}_{\al\be}r^{\al}x^{\be}_{\ ;k}\bar{R}_{\al\be\g\de}\nu^{\al}x^{\be}_{\ ;i}g^{km}x^{\g}_{\ ;m}x^{\de}_{\ ;j}.}
Note 
\eq{\bar{g}_{\al\be}r^{\al}x^{\be}_{\ ;k}=r_{;k},}
we thus get
\eq{\label{eq-u}u_{;ij}=\la'h_{ij}+ \la r_{;k}{h_{ij;}}^k- h_{i}^k h_{kj}u-\la\bar{R}_{\al\be\g\de}\nu^{\al}x^{\be}_{\ ;i}x^{\g}_{\ ;m}x^{\de}_{\ ;j}{r_{;}}^m.}
Since
\eq{\cF_{;k}=-\fr{n}{F^2}F^{ij}h_{ij;k}-\fr{u_{;k}}{\la'}+\fr{\la''u}{\la'^2}r_{;k},}
we obtain \eqref{Ev-u}.

From \eqref{EV-W}, we have
\eq{\dot{F}&=-F^{ij}\cF_{;ij}-F^{ij}\cF h_{jk}h^k_i-F^{ij}\cF\bar{R}_{\al\be\g\de}x^{\al}_{\ ;i}\nu^{\be}\nu^{\g}x^{\de}_{\ ;j}
\\&=\frac{n}{F^2}F^{ij}F_{;ij}-\frac{2n}{F^3}F^{ij}F_{;i}F_{;j}+\frac{1}{\la'}F^{ij}u_{;ij}-\frac{u}{\la'^2}F^{ij}\la'_{;ij}+\frac{2u}{\la'^3}F^{ij}\la'_{;i}\la'_{;j}\\
	&\hp{=}-\frac{2\la''}{\la'^2}F^{ij}u_{;i}r_{;j}-\left(\frac{n}{F}-\frac{u}{\la'}\right)F^{ij}h_{jk}h^k_i-\left(\frac{n}{F}-\frac{u}{\la'}\right)F^{ij}\bar{R}_{\al\be\g\de}x^{\al}_{\ ;i}\nu^{\be}\nu^{\g}x^{\de}_{\ ;j}.
}
Using \eqref{eq-u} and \eqref{graph-h}, we get \eqref{EV-F}.}

We also need the parabolic equation satisfied by the Weingarten operator. A similar calculation was performed in \cite[Lemma~2.4.1]{Gerhardt:/2006}, but since our flow speed is not directly covered by this reference, we deduce it for convenience.

\begin{lemma}\label{EV-W-B}
Along \eqref{Flow} the following evolution equation holds.
\eq{\label{ev-h}\cL h^j_i&=-\fr{2n}{F^3}F_{;i}{F_{;}}^j+\fr{n}{F^2}F^{kl,rs}h_{kl;i}{h_{rs;}}^j-\fr{\la''}{\la'^2}\br{u_{;i}{r_{;}}^j+r_{;i}{u_{;}}^j}\\  
				&\hp{=}-\fr{u}{\la'^2}\br{\la'''-\fr{2\la''^2}{\la'}-\fr{\la''\la'}{\la}}r_{;i}{r_{;}}^j+\br{1+\fr{u\la''}{\la'^2 v}}h^j_i\\
		&\hp{=}-\fr{u\la''}{\la'\la}\de^j_i-\fr{\la}{\la'}\bar{R}_{\al\be\g\de}\nu^{\al}x^{\be}_{\ ;i}x^{\g}_{\ ;m}x^{\de}_{\ ;l}{r_{;}}^mg^{lj}\\
        &\hp{=}-\br{\fr{n}{F}-\fr{u}{\la'}}\bar{R}_{\al\be\g\de}x^{\al}_{\ ;i}\nu^{\be}\nu^{\g}x^{\de}_{\ ;m}g^{mj}+\fr{n}{F^2}F^{kl}h_{rk}h^r_l h_{i}^j-\fr{2n}{F}h_{k}^jh^k_i\\
                &\hp{=}+\fr{n}{F^2}F^{kl}\bar{R}_{\al\be\g\de}\br{x^{\al}_{\ ;l}x^{\be}_{\ ;r}x^{\g}_{\ ;k}x^{\de}_{\ ;m}h^m_i+x^{\al}_{\ ;l}x^{\be}_{\ ;i}x^{\g}_{\ ;k}x^{\de}_{\ ;m}h^m_r}g^{rj}\\
                &\hp{=}+\fr{2n}{F^2}F^{kl}\bar{R}_{\al\be\g\de}x^{\al}_{\ ;l}x^{\be}_{\ ;r}x^{\g}_{\ ;i}x^{\de}_{\ ;m}h^m_k g^{rj}+\fr{n}{F^2}F^{kl}\bar{R}_{\al\be\g\de}\nu^{\al}x^{\be}_{\ ;k}x^{\g}_{\ ;l}\nu^{\de}h_{i}^j\\
                &\hp{=}+\fr{n}{F}\bar{R}_{\al\be\g\de}\nu^{\al}x^{\be}_{\ ;i}\nu^{\g}x^{\de}_{\ ;m}g^{mj}-\fr{n}{F^2}F^{kl}\bar{R}_{\al\be\g\de;\e}\nu^{\al}x^{\be}_{\ ;k}x^{\g}_{\ ;l}x^{\de}_{\ ;i}x^{\e}_{\ ;m}g^{mj}\\
                &\hp{=}-\fr{n}{F^2}F^{kl}\bar{R}_{\al\be\g\de;\e}\nu^{\al}x^{\be}_{\ ;i}x^{\g}_{\ ;k}x^{\de}_{\ ;m}x^{\e}_{\ ;l} g^{mj}.}
\end{lemma}

\pf{We use \eqref{Ev-h} and calculate $-\cF_{;ij}$ step by step.
We use
\eq{\cF=\fr{n}{F}-\fr{u}{\la'},\quad -\cF_{;i}=\fr{n}{F^2}F_{;i}+\fr{u_{;i}}{\la'}-\fr{u\la''r_{;i}}{{\la'}^2},}
\eqref{eq-u} as well as \eqref{graph-h}, to deduce
\eq{\label{Ev-W-B-1}-\cF_{;ij}&=-\fr{2n}{F^3}F_{;i}F_{;j}+\fr{n}{F^2}F_{;ij}+\fr{u_{;ij}}{\la'}-\fr{\la''}{\la'^2}\br{u_{;i}r_{;j}+r_{;i}u_{;j}}\\
				&\hp{=}-u\br{\fr{\la'''}{\la'^2}-\fr{2\la''^2}{\la'^3}}r_{;i}r_{;j}-\fr{u\la''}{\la'^2}r_{;ij}\\
                &=-\fr{2n}{F^3}F_{;i}F_{;j}+\fr{n}{F^2}F_{;ij}-\fr{\la''}{\la'^2}\br{u_{;i}r_{;j}+r_{;i}u_{;j}}\\
                &\hp{=}-u\br{\fr{\la'''}{\la'^2}-\fr{2\la''^2}{\la'^3}}r_{;i}r_{;j}+h_{ij}+\fr{\la}{\la'}r_{;k}{h_{ij;}}^k-\fr{u}{\la'}h_{ik}h^k_j\\
                &\hp{=}-\fr{\la}{\la'}\bar{R}_{\al\be\g\de}\nu^{\al}x^{\be}_{\ ;i}x^{\g}_{\ ;m}x^{\de}_{\ ;j}{r_{;}}^m+\fr{u\la''}{\la'^2}\br{v^{-1}h_{ij}-\fr{\la'}{\la}g_{ij}+\fr{\la'}{\la}r_{;i}r_{;j}}\\
                &=-\fr{2n}{F^3}F_{;i}F_{;j}+\fr{n}{F^2}F_{;ij}-\fr{\la''}{\la'^2}\br{u_{;i}r_{;j}+r_{;i}u_{;j}}+\fr{\la}{\la'}r_{;k}{h_{ij;}}^k\\
                &\hp{=}-\fr{u}{\la'^2}\br{\la'''-\fr{2\la''^2}{\la'}-\fr{\la''\la'}{\la}}r_{;i}r_{;j}+\br{1+\fr{u\la''}{\la'^2 v}}h_{ij}-\fr{u}{\la'}h_{ik}h^k_j\\
                &\hp{=}-\fr{u\la''}{\la'\la}g_{ij}-\fr{\la}{\la'}\bar{R}_{\al\be\g\de}\nu^{\al}x^{\be}_{\ ;i}x^{\g}_{\ ;m}x^{\de}_{\ ;j}{r_{;}}^m.}
We have to transform $F_{;ij}.$ Using the Codazzi equation \eqref{Codazzi}and the Ricci identities \eqref{Ricci}, we obtain
\eq{F_{;ij}&=F^{kl,rs}h_{kl;i}h_{rs;j}+F^{kl}h_{kl;ij}\\
				&=F^{kl,rs}h_{kl;i}h_{rs;j}+F^{kl}h_{ki;lj}-F^{kl}\br{\bar{R}_{\al\be\g\de}\nu^{\al}x^{\be}_{\ ;k}x^{\g}_{\ ;l}x^{\de}_{\ ;i}}_{;j}\\
                &=F^{kl,rs}h_{kl;i}h_{rs;j}+F^{kl}h_{ki;jl}+F^{kl}{R_{ljk}}^ah_{ai}+F^{kl}{R_{lji}}^a h_{ka}\\
                &\hp{=}-F^{kl}\br{\bar{R}_{\al\be\g\de}\nu^{\al}x^{\be}_{\ ;k}x^{\g}_{\ ;l}x^{\de}_{\ ;i}}_{;j}\\
                &=F^{kl,rs}h_{kl;i}h_{rs;j}+F^{kl}{R_{ljk}}^ah_{ai}+F^{kl}{R_{lji}}^a h_{ka}+F^{kl}h_{ij;kl}\\
                &\hp{=}-F^{kl}\br{\bar{R}_{\al\be\g\de}\nu^{\al}x^{\be}_{\ ;k}x^{\g}_{\ ;l}x^{\de}_{\ ;i}}_{;j}-F^{kl}\br{\bar{R}_{\al\be\g\de}\nu^{\al}x^{\be}_{\ ;i}x^{\g}_{\ ;k}x^{\de}_{\ ;j}}_{;l}}
 Differentiating the big brackets by the product rule gives, using the Weingarten equation \eqref{Weingarten} and the Gauss equation \eqref{GE}
 \eq{F_{;ij}&=F^{kl}h_{ij;kl}+F^{kl,rs}h_{kl;i}h_{rs;j}\\
                &\hp{=}+F^{kl}(h_{la}h_{jk}-h_{lk}h_{ja}+\bar R_{\al\be\g\de}x^{\al}_{\ ;l}x^{\be}_{\ ;j}x^{\g}_{\ ;k}x^{\de}_{\ ;a})h_{i}^a\\   
                &\hp{=}+F^{kl}(h_{la}h_{ji}-h_{li}h_{ja}+\bar R_{\al\be\g\de}x^{\al}_{\ ;l}x^{\be}_{\ ;j}x^{\g}_{\ ;i}x^{\de}_{\ ;a})h_{k}^a\\
             &\hp{=}-F^{kl}\bar{R}_{\al\be\g\de;\e}\nu^{\al}x^{\be}_{\ ;k}x^{\g}_{\ ;l}x^{\de}_{\ ;i}x^{\e}_{\ ;j}-F^{kl}\bar{R}_{\al\be\g\de}x^{\al}_{\ ;m}x^{\be}_{\ ;k}x^{\g}_{\ ;l}x^{\de}_{\ ;i}h^m_j\\
                &\hp{=}+F^{kl}\bar{R}_{\al\be\g\de}\nu^{\al}x^{\be}_{\ ;k}\nu^{\g}x^{\de}_{\ ;i}h_{lj}+F^{kl}\bar{R}_{\al\be\g\de}\nu^{\al}x^{\be}_{\ ;k}x^{\g}_{\ ;l}\nu^{\de}h_{ij} \\
                &\hp{=}-F^{kl}\bar{R}_{\al\be\g\de;\e}\nu^{\al}x^{\be}_{\ ;i}x^{\g}_{\ ;k}x^{\de}_{\ ;j}x^{\e}_{\ ;l}-F^{kl}\bar{R}_{\al\be\g\de}x^{\al}_{\ ;m}x^{\be}_{\ ;i}x^{\g}_{\ ;k}x^{\de}_{\ ;j}h^m_l \\
                &\hp{=}+F^{kl}h_{kl}\bar{R}_{\al\be\g\de}\nu^{\al}x^{\be}_{\ ;i}\nu^{\g}x^{\de}_{\ ;j}+F^{kl}\bar{R}_{\al\be\g\de}\nu^{\al}x^{\be}_{\ ;i}x^{\g}_{\ ;k}\nu^{\de}h_{jl}}
and after some rearranging, using the homogeneity of $F$,
\eq{\label{Ev-W-B-3}  F_{;ij} &=F^{kl}h_{ij;kl}+F^{kl,rs}h_{kl;i}h_{rs;j}+F^{kl}h_{rk}h^r_l h_{ij}-F h_{ik}h^k_j\\
                &\hp{=}+F^{kl}\bar{R}_{\al\be\g\de}\br{x^{\al}_{\ ;l}x^{\be}_{\ ;j}x^{\g}_{\ ;k}x^{\de}_{\ ;m}h^m_i+x^{\al}_{\ ;l}x^{\be}_{\ ;i}x^{\g}_{\ ;k}x^{\de}_{\ ;m}h^m_j}\\
                &\hp{=}+2F^{kl}\bar{R}_{\al\be\g\de}x^{\al}_{\ ;l}x^{\be}_{\ ;j}x^{\g}_{\ ;i}x^{\de}_{\ ;m}h^m_k+F^{kl}\bar{R}_{\al\be\g\de}\nu^{\al}x^{\be}_{\ ;k}x^{\g}_{\ ;l}\nu^{\de}h_{ij}\\
                &\hp{=}+F\bar{R}_{\al\be\g\de}\nu^{\al}x^{\be}_{\ ;i}\nu^{\g}x^{\de}_{\ ;j}-F^{kl}\bar{R}_{\al\be\g\de;\e}\nu^{\al}x^{\be}_{\ ;k}x^{\g}_{\ ;l}x^{\de}_{\ ;i}x^{\e}_{\ ;j}\\
                &\hp{=}-F^{kl}\bar{R}_{\al\be\g\de;\e}\nu^{\al}x^{\be}_{\ ;i}x^{\g}_{\ ;k}x^{\de}_{\ ;j}x^{\e}_{\ ;l}.
                }
From \eqref{Ev-h}, inserting \eqref{Ev-W-B-1}, we get 

\eq{\dot{h}_{ij}&=-\cF_{;ij}+\cF h_{ik}h^k_j-\cF\bar{R}_{\al\be\g\de}x^{\al}_{\ ;i}\nu^{\be}\nu^{\g}x^{\de}_{\ ;j}\\
				&=-\fr{2n}{F^3}F_{;i}F_{;j}+\fr{n}{F^2}F_{;ij}-\fr{\la''}{\la'^2}\br{u_{;i}r_{;j}+r_{;i}u_{;j}}+\fr{\la}{\la'}r_{;k}{h_{ij;}}^k\\
                &\hp{=}-\fr{u}{\la'^2}\br{\la'''-\fr{2\la''^2}{\la'}-\fr{\la''\la'}{\la}}r_{;i}r_{;j}+\br{1+\fr{u\la''}{\la'^2 v}}h_{ij}-\fr{u}{\la'}h_{ik}h^k_j-\fr{u\la''}{\la'\la}g_{ij}\\
                &\hp{=}-\fr{\la}{\la'}\bar{R}_{\al\be\g\de}\nu^{\al}x^{\be}_{\ ;i}x^{\g}_{\ ;m}x^{\de}_{\ ;j}{r_{;}}^m+\br{\fr{n}{F}-\fr{u}{\la'}} h_{ik}h^k_j\\
                &\hp{=}-\br{\fr{n}{F}-\fr{u}{\la'}}\bar{R}_{\al\be\g\de}x^{\al}_{\ ;i}\nu^{\be}\nu^{\g}x^{\de}_{\ ;j}.}
Inserting \eqref{Ev-W-B-3} into this equation gives
\eq{\dot{h}_{ij}&=\fr{n}{F^2}F^{kl}h_{ij;kl}+\fr{\la}{\la'}r_{;k}{h_{ij;}}^k-\fr{2n}{F^3}F_{;i}F_{;j}+\fr{n}{F^2}F^{kl,rs}h_{kl;i}h_{rs;j}\\
		&\hp{=}-\fr{\la''}{\la'^2}\br{u_{;i}r_{;j}+r_{;i}u_{;j}}    -\fr{u}{\la'^2}\br{\la'''-\fr{2\la''^2}{\la'}-\fr{\la''\la'}{\la}}r_{;i}r_{;j}+\br{1+\fr{u\la''}{\la'^2 v}}h_{ij}\\
		&\hp{=}-\fr{2u}{\la'}h_{ik}h^k_j-\fr{u\la''}{\la'\la}g_{ij}-\fr{\la}{\la'}\bar{R}_{\al\be\g\de}\nu^{\al}x^{\be}_{\ ;i}x^{\g}_{\ ;m}x^{\de}_{\ ;j}{r_{;}}^m\\
                &\hp{=}-\br{\fr{n}{F}-\fr{u}{\la'}}\bar{R}_{\al\be\g\de}x^{\al}_{\ ;i}\nu^{\be}\nu^{\g}x^{\de}_{\ ;j}+\fr{n}{F^2}F^{kl}h_{rk}h^r_l h_{ij}\\
                &\hp{=}+\fr{n}{F^2}F^{kl}\bar{R}_{\al\be\g\de}\br{x^{\al}_{\ ;l}x^{\be}_{\ ;j}x^{\g}_{\ ;k}x^{\de}_{\ ;m}h^m_i+x^{\al}_{\ ;l}x^{\be}_{\ ;i}x^{\g}_{\ ;k}x^{\de}_{\ ;m}h^m_j}\\
                &\hp{=}+\fr{2n}{F^2}F^{kl}\bar{R}_{\al\be\g\de}x^{\al}_{\ ;l}x^{\be}_{\ ;j}x^{\g}_{\ ;i}x^{\de}_{\ ;m}h^m_k+\fr{n}{F^2}F^{kl}\bar{R}_{\al\be\g\de}\nu^{\al}x^{\be}_{\ ;k}x^{\g}_{\ ;l}\nu^{\de}h_{ij}\\
                &\hp{=}+\fr{n}{F}\bar{R}_{\al\be\g\de}\nu^{\al}x^{\be}_{\ ;i}\nu^{\g}x^{\de}_{\ ;j}-\fr{n}{F^2}F^{kl}\bar{R}_{\al\be\g\de;\e}\nu^{\al}x^{\be}_{\ ;k}x^{\g}_{\ ;l}x^{\de}_{\ ;i}x^{\e}_{\ ;j}\\
                &\hp{=}-\fr{n}{F^2}F^{kl}\bar{R}_{\al\be\g\de;\e}\nu^{\al}x^{\be}_{\ ;i}x^{\g}_{\ ;k}x^{\de}_{\ ;j}x^{\e}_{\ ;l}.}
Using
\eq{\dot{h}^i_j=\dot{g}^{ik}h_{kj}+g^{ik}\dot{h}_{kj}=-g^{il}\dot{g}_{lm}g^{mk}h_{kj}+g^{ik}\dot{h}_{kj}=2\br{\fr{u}{\la'}-\fr{n}{F}}h^i_k h^k_{j}+g^{ik}\dot{h}_{kj}}
gives the result.
}

In particular, when the ambient space $N$ is a space form of sectional curvature $K_N$, then 
\eq{\bar R_{\al\be\g\de}=K_N(\bar g_{\al\de}\bar g_{\be\g}-\bar g_{\al\g}\bar g_{\be\de})} and
\eq{\la''=-K_N\la, \quad \la'''=-K_N\la'=\frac{\la''\la'}{\la}} and \eqref{ev-h} reduces to
\eq{\label{Ev-h-spaceform}\cL h^j_i&=-\fr{2n}{F^3}F_{;i}{F_{;}}^j+\fr{n}{F^2}F^{kl,rs}h_{kl;i}{h_{rs;}}^j-\fr{\la''}{\la'^2}\br{u_{;i}{r_{;}}^j+r_{;i}{u_{;}}^j}\\  
				&\hp{=}-\fr{u}{\la'^2}\br{\la'''-\fr{2\la''^2}{\la'}-\fr{\la''\la'}{\la}}r_{;i}{r_{;}}^j+\br{1+\fr{u\la''}{\la'^2 v}}h^j_i\\
		&\hp{=}-\fr{u\la''}{\la'\la}\de^j_i-\br{\fr{n}{F}-\fr{u}{\la'}}K_N\de_i^j+\fr{n}{F^2}F^{kl}h_{rk}h^r_l h_{i}^j-\fr{2n}{F}h_{k}^jh^k_i\\
                &\hp{=}+\fr{n}{F^2}F^{kl}K_N(h_{il}\de_{k}^j+h_l^j  g_{ik}-2g_{kl}h_i^j)\\
                &\hp{=}+\fr{2n}{F^2}F^{kl}K_N(h_{kl}\de_i^j-g_{li}h_k^j)+K_N\fr{n}{F^2}F^{kl}g_{kl}h_{i}^j-K_N\fr{n}{F}\de_i^j
                \\&=-\fr{2n}{F^3}F_{;i}{F_{;}}^j+\fr{n}{F^2}F^{kl,rs}h_{kl;i}{h_{rs;}}^j+K_N\fr{\la}{\la'^2}\br{u_{;i}{r_{;}}^j+r_{;i}{u_{;}}^j}\\  
				&\hp{=}+K^2_N\fr{2u\la^2}{\la'^3}r_{;i}{r_{;}}^j+\br{1-K_N\fr{u^2}{\la'^2}+\fr{n}{F^2}F^{kl}h_{rk}h^r_l-\fr{n}{F^2}K_NF^{kl}g_{kl}}h^j_i\\
        &\hp{=}+2\fr{u}{\la'}K_N\de_i^j-\fr{2n}{F}h_{k}^jh^k_i.
                }

\medskip

\section{Upper bounds for the curvature function}\label{sec:Bound-F}
In this section we  show that the curvature function $F$ is bounded from above along the flow \eqref{Flow} in the case $F=n\frac{H_{k}}{H_{k-1}}$ for very general $\la$. For this paper, we only apply it in the case $\la=\sin$, but due to its generality it might be of use in further situations.

\begin{prop}\label{F-bound}
Let $a,b\in\R$ and $(N,\bar g)$ be the warped space $((a,b)\x \bbS^n,dr^2+\la^2(r)\s)$ with $\la,\la'>0$. Let 
\eq{F=n\fr{H_k}{H_{k-1}}}
and let $x_0(M)$be the embedding of a closed $n$-dimensional manifold $M$ into $N$, such that $x_0(M)$ is a graph over the domain $\bbS^n$ and such that $\ka\in \G_k$ for all $n$-tupels of principal curvatures along $x_0(M)$. Then along any solution $x$ of \eqref{Flow} with initial embedding $x_0$ there exists a constant $c=c(n,k,\sup r_0,\inf r_0,\la)$, such that
\eq{F\leq c.}
\end{prop}

\begin{rem}\label{Barriers}
Note that under the hypothesis of \cref{F-bound} we have
\eq{r\leq \sup r_0,\q r\geq \inf r_0}
along the flow, due to the maximum principle. This assertion also holds for arbitrary monotone curvature functions $F$. 
\end{rem}
Now we prove \cref{F-bound}.
\pf{
Consider the test function 
\eq{\Phi= \log F+\frac{u}{\la}+\al r} with a large constant $\al$ to be determined.
Assume $\Phi$ attains its maximum at $p$. By a suitable choice of coordinate we can assume $g_{ij}|_p=\delta_{ij}$, $h_{ij}|_p$ is diagonal and in turn $F^{ij}$ is diagonal at $p$. Assume $F|_p\ge C$ for some sufficient large constant $C$. In the following, we compute at $p$.

From \Cref{Ev-Eq} we deduce 
\eq{\mathcal{L}\left(\frac{u}{\la}\right)&=\frac{1}{\la}\mathcal{L}u-\frac{u\la'}{\la^2}\mathcal{L}r+\frac{n}{F^2}\frac{2\la'}{\la^2}F^{ij}u_{;i}r_{;j}-\frac{n}{F^2}\frac{u}{\la^3}\left(2\la'^2-\la\la''\right)F^{ij}r_{;i}r_{;j}
\\&= \fr{n}{F^2}\br{F^{ij}h_{ik}h^k_j-\fr 1n F^2}\bar g(\partial_r, \nu)-\fr{\la''}{\la'^2}\|\n r\|^2u
\\&\quad -\frac{u\la'}{\la^2}\left(\frac{2n}{vF}-\frac{\la}{\la'}- \fr{n\la'}{\la F^2}F^{ij}g_{ij}+\fr{n\la'}{\la F^2}F^{ij}r_{;i}r_{;j}\right)
\\&\quad +\frac{n}{F^2}\frac{2\la'}{\la^2}F^{ij}u_{;i}r_{;j}-\frac{n}{F^2}\frac{u}{\la^3}\left(2\la'^2-\la\la''\right)F^{ij}r_{;i}r_{;j}\\&\quad +\fr{n}{F^2}F^{ij}\bar{R}_{\al\be\g\de}\nu^{\al}x^{\be}_{\ ;i}x^{\g}_{\ ;m}x^{\de}_{\ ;j}{r_{;}}^m
\\&\le \fr{n}{F^2}\br{F^{ii}h_{ii}^2-\fr 1n F^2}\bar g(\partial_r, \nu)
 +\frac{C}{F}F^{ii}|u_{;i}||r_{;i}|+ C+\sum_i CF^{ii}.
}

We also have
\eq{\mathcal{L}\log F&=-\frac{n}{F^2}F^{ij}(\log F)_{;i}(\log F)_{;j}-\frac{n}{F^2}\left(F^{ij}h_{jk}h^k_i-\frac1n F^2\right)+\frac{u^2\la''}{\la\la'^2}
\\&\quad -\frac{u\la''}{\la\la'F}F^{ij}g_{ij}+\frac{u}{\la'^2F}\left(\frac{\la'\la''}{\la}-\la'''+\frac{2{\la''}^2}{\la'}\right)F^{ij}r_{;i}r_{;j}-\frac{2\la''}{\la'^2F}F^{ij}u_{;i} r_{;j}
\\&\quad -\frac{\la}{\la'F}F^{ij}\bar{R}_{\al\be\g\de}\nu^{\al}x^{\be}_{\ ;i}x^{\g}_{\ ;m}x^{\de}_{\ ;j}{r_{;}}^m-\frac 1F\left(\frac{n}{F}-\frac{u}{\la'}\right)F^{ij}\bar{R}_{\al\be\g\de}x^{\al}_{\ ;i}\nu^{\be}\nu^{\g}x^{\de}_{\ ;j}
\\&\le -\frac{n}{F^2}F^{ii}(\log F)_{;i}(\log F)_{;i}-\frac{n}{F^2}\left(F^{ii}h_{ii}^2-\frac1n F^2\right)\\&\quad+\frac{C}{F}F^{ii}|u_{;i}||r_{;i}|+C+\sum_i CF^{ii}.
}

Thus 
\eq{\mathcal{L}\Phi&=\mathcal{L}\log F+\mathcal{L}\left(\frac{u}{\la}\right)+ \al \mathcal{L}r
\\&\le -\frac{n}{F^2}F^{ii}(\log F)_{;i}(\log F)_{;i}-\frac{n}{F^2}\left(F^{ii}h_{ii}^2-\frac1n F^2\right)(1-\bar g(\partial_r, \nu))
\\&\quad -\al \frac{\la}{\la'}+\al \frac{C}{F}+\al\frac{C}{F}\sum_i F^{ii} +\frac{C}{F}F^{ii}|u_{;i}||r_{;i}|+C+CF^{ii}.
}

A calculation using the Newton-MacLaurin inequalities gives for our special $F$: 
\eq{F^{ii}h_{ii}^2-\frac1n F^2\ge 0} and
\eq{F^{i}_i\le C(n,k),}
see \cite[Prop.~2.2]{HuiskenSinestrari:09/1999} for useful formulas for this calculation.

Thus 
\eq{\mathcal{L}\Phi&\le -\frac{n}{F^2}F^{ii}(\log F)_{;i}(\log F)_{;i}-\al \frac{\la}{\la'}+\al \frac{C}{F} +\frac{C}{F}F^{ii}|u_{;i}||r_{;i}|+C.
}

From the maximal property of $\Phi$ at $p$, we have 
\eq{\n \log F=-\frac{1}{\la}\n u+\frac{u\la'}{\la^2}\n r-\al \n r.}

Therefore
\eq{0\le \cL\Phi &\le -\frac{n}{F^2}F^{ii}\left(-\frac{1}{\la}u_{;i}+\frac{u\la'}{\la^2}r_{;i}-\al r_{;i}\right)^2-\al \frac{\la}{\la'}+\al \frac{C}{F} +\frac{C}{F}F^{ii}|u_{;i}||r_{;i}|+C
\\&\le -\frac{n}{2F^2\la^2}F^{ii}{u_{;i}}^2+ \frac{n}{F^2}F^{ii}\left(\frac{u\la'}{\la^2}r_{;i}-\al r_{;i}\right)^2+\frac{C}{F}F^{ii}|u_{;i}||r_{;i}|\\
	&\hp{=}-\al \frac{\la}{\la'}+\al \frac{C}{F}+C
\\&\le -\frac{n}{2F^2\la^2}F^{ii}\left(|u_{;i}|-\frac{CF\la^2}{n}|r_{;i}|\right)^2-\al \frac{\la}{\la'}+\al \frac{C}{F}+C+C\frac{\al^2}{F^2}
\\ &\le -\al \frac{\la}{\la'}+\al \frac{C}{F}+C+C\frac{\al^2}{F^2}.
}

Assume $F|_p\ge \al$. Then by choosing $\al$ large enough, we get the RHS of above inequality is negative, a contradiction. Therefore, $F|_p\le \al$ for our choice of $\al$ and in turn $\Phi|_p$ is bounded. Since $\Phi$ attains its maximum at $p$, we conclude that $F$ is bounded from above.
}

\medskip 

\section{Gradient estimates}\label{sec:Grad-Bound}
In this section we  show that the graph function has a uniform $C^1$ bound along the flow \eqref{Flow} for very general $F$ and $\la$.

\begin{prop}\label{grad-bound}
Let $a,b\in\R$ and $(N,\bar g)$ be the warped space $((a,b)\x \bbS^n,dr^2+\la^2(r)\s)$ with $\la,\la'>0$. Let $F\in C^{\8}(\G)$ be a $1$-homogeneous and strictly monotone curvature function and let $x_0(M)$ be the embedding of a closed $n$-dimensional manifold $M$ into $N$, such that $x_0(M)$ is a graph over the domain $\bbS^n$ with graph function $r$ and such that $\ka\in \G$ for all $n$-tuples of principal curvatures along $x_{0}(M)$. 
Then along any solution $x$ of \eqref{Flow} with initial embedding $x_0$, there exists a constant $c=c(n,\sup r_0,\inf r_0,\la)$, such that
\eq{|\hat{\nabla} r|\leq c.}
\end{prop}

\begin{proof}
Recall \eq{\label{phi}\varphi=\int_a^r \frac{1}{\la(s)}~ds} To simplify the notation, we just use $\varphi_i=\varphi_{,i}$, etc., i.e., we omit the comma when taking covariant derivative on $\mathbb{S}^n$.

We rewrite the flow equation as a scalar equation on $\varphi$:
\eq{\label{gamma0}
\del_t\varphi&=\frac{1}{\la}\left(\frac{n}{F\left(\frac{\la'}{\la v}\de_i^j-\frac{1}{\la v}\tilde{g}^{jk}\varphi_{ki}\right)}-\frac{u}{\la'}\right)v\\&=\frac{n v^2}{F(\la'\de_i^j- \tilde{g}^{jk}\varphi_{ki})}-\frac{1}{\la'}=:G(\varphi, \hat{\n} \varphi, \hat{\n}^2 \varphi),}
where $\tilde{g}^{ij}=\s^{ij}-\frac{\varphi^i\varphi^j}{v^2}.$ For simplicity, we denote by $F=F(\la'\de_i^j- \tilde{g}^{jk}\varphi_{ki})$ and $F^i_j$ the derivative of $F$ with respect to its argument.

We compute
\eq{\label{G-phi} &G^{ij}:=\frac{\del G}{\del \varphi_{ij}}= \frac{nv^2}{F^2}F_k^{i}\tilde{g}^{kj},\\
&G^{\varphi_p}:=\frac{\del G}{\del \varphi_p}=\frac{2n \varphi^p}{F}+\frac{nv^2}{F^2}F^i_j  \left(-\frac{\s^{jp}\varphi^k+\s^{kp}\varphi^j}{v^2}+\frac{2\varphi^j\varphi^k\varphi^p}{v^4}\right)\varphi_{ki},
\\&G^\varphi:=\frac{\del G}{\del \varphi}=-\frac{nv^2\la''\la}{F^2}F^i_i +\frac{\la''\la}{{\la'}^2}.
}
Using the $1$-homogeneity of $F$, we have 
\eq{\label{C1eq1}
G^{ij}\varphi_{ij}&=\frac{nv^2}{F^2}F_k^{i}\tilde{g}^{kj}\varphi_{ij}\\&=-\frac{nv^2}{F^2}F_k^{i}(\la'\de_i^k-\tilde{g}^{kj}\varphi_{ij})+\frac{nv^2\la'}{F^2}F_i^{i}=-\frac{nv^2}{F}+\frac{nv^2\la'}{F^2}F_i^{i}}
and
\eq{\label{C1eq2} G^{ij}\varphi_i\varphi_j &=\frac{nv^2}{F^2}F_k^{i}\tilde{g}^{kj}\varphi_i\varphi_j=\frac{n}{F^2}F_k^{i}\varphi^k\varphi_i.}

Let $\mathcal{L}=\p_t-G^{ij}\n^2_{ij}$ be the parabolic operator. Using the Ricci identities on $\bbS^n$ we get

\eq{\label{C1eq3}\cL|\hat\n\p|^2&=-2G^{ij}\varphi_{ik}\varphi_{j}^{\ k}-2G^{ij}\s_{ij}|\hat{\n}\varphi|^2+2G^{ij}\varphi_i\varphi_j\\
					&\hp{=}+G^{\varphi_p}(|\hat{\n}\varphi|^2)_{p}+2G^\varphi|\hat{\n}\varphi|^2.}

Let $f: [0, \infty)\to (0, \infty)$ be an auxiliary function to be determined. 
Consider a test function $$\Phi=\log \frac{|\hat{\n}\varphi|^2}{f(\varphi)}.$$ In the following we compute at a maximal point of $\Phi$. Due to the maximal property, \eq{\hat{\n}|\hat{\n}\varphi|^2=\frac{f'}{f}|\hat{\n}\varphi|^2\hat{\n} \varphi.} 
By a suitable choice of the coordinates, we may assume $\s_{ij}=\de_{ij}$ and $|\hat{\n}\varphi|=\varphi_1$. Then \eq{\varphi_{11}=\frac12\frac{f'}{f}|\hat{\n}\varphi|^2,\q \varphi_{1j}=0~\mbox{for}~j=2,\cdots, n.}
Then $\tilde{g}^{ij}$ is diagonal with 
\eq{\tilde{g}^{11}=\frac{1}{v^2},\q \tilde{g}^{ii}=1~\mbox{for} ~i\neq 1.} We may further assume $\varphi_{ij}$ is diagonal and in turn $F_{i}^k$ is diagonal. Thus
we have
\eq{\label{C1eq4} -2G^{ij}\varphi_{ik}\varphi_{j}^{\ k}&=-\frac{2nv^2}{F^2}F_l^{i}\tilde{g}^{lj}\varphi_{ik}\varphi_{j}^{\ k}
\\&= -\frac{2n}{F^2}F^{11}\frac14\left(\frac{f'}{f}\right)^2|\hat{\n}\varphi|^4-\frac{2nv^2}{F^2}\sum_{k\ge 2}F^{kk}\varphi_{kk}^2,}

\eq{\label{C1eq5} -2G^{ij}\s_{ij}|\n\varphi|^2+2G^{ij}\varphi_i\varphi_j&=- \frac{2nv^2}{F^2}F_k^{i}\ti{g}^{kj}(\s_{ij}|\hat{\n}\varphi|^2-\varphi_i\varphi_j)\\&= - \frac{2nv^2}{F^2}\sum_{k\ge 2}F^{kk}|\hat{\n}\varphi|^2,}

\eq{\label{C1eq6}
G^{\varphi_p}(|\hat{\n}\varphi|^2)_{p}
&=\left(\frac{2n\varphi^p}{F}+\frac{nv^2}{F^2}F^i_j \br{-\frac{\s^{jp}\varphi^k+\s^{kp}\varphi^j}{v^2}+\frac{2\varphi^j\varphi^k\varphi^p}{v^4}}\varphi_{ki}\right)\frac{f'}{f}|\hat{\n}\varphi|^2\varphi_p
\\&=\frac{f'}{f}\frac{2n}{F}|\hat{\n}\varphi|^4-\left(\frac{f'}{f}\right)^2\frac{n}{v^2F^2}F^{11}|\hat{\n}\varphi|^6}

and

\eq{\label{C1eq7} 2G^\varphi|\hat{\n}\varphi|^2&=\left(-\frac{2nv^2\la''\la}{F^2}F^{i}_i +\frac{2\la''\la}{{\la'}^2}\right)|\hat{\n}\varphi|^2.}

On the other hand, using \eqref{gamma0}, \eqref{C1eq1} and \eqref{C1eq2} we get
\eq{\label{C1eq8}
\mathcal{L} (f(\varphi))&=f'\left(\frac{n v^2}{F}-\frac{1}{\la'}\right)-f' G^{ij}\varphi_{ij}-f''G^{ij}\varphi_i\varphi_j
\\&=f'\left(\frac{2n v^2}{F} -\frac{1}{\la'}\right)-f' \frac{nv^2\la'}{F^2}F^{i}_i-f''\frac{n}{F^2}F^{11}|\hat{\n}\varphi|^2.}

Using \eqref{C1eq3}--\eqref{C1eq8} and the maximal property of $\Phi$ at $p$,  we have
\eq{\label{C1eq9}
0&\le \frac{\mathcal{L}(|\hat{\n}\varphi|^2)}{|\n\varphi|^2}-\frac{\mathcal{L}f}{f}
\\&= -\frac{2n}{F^2}F^{11}\frac14\left(\frac{f'}{f}\right)^2|\hat{\n}\varphi|^2-\frac{2nv^2}{F^2|\hat{\n}\varphi|^2}\sum_{k\ge 2}F^{kk}\varphi_{kk}^2- \frac{2nv^2}{F^2}\sum_{k\ge 2}F^{kk}
\\&\quad +\frac{f'}{f}\frac{2n}{F}|\hat{\n}\varphi|^2-\left(\frac{f'}{f}\right)^2\frac{n}{v^2F^2}F^{11}|\hat{\n}\varphi|^4-\frac{2nv^2\la''\la}{F^2}F^{i}_i+2\frac{\la''\la}{{\la'}^2}
\\&\quad -\frac{f'}{f}\left(\frac{2n v^2}{F} -\frac{1}{\la'}\right)+\frac{f'}{f} \frac{nv^2\la'}{F^2}F^{i}_i +\frac{f''}{f}\frac{n}{F^2}F^{11}|\hat{\n}\varphi|^2
\\&= -\frac{2n}{F^2}F^{11}\frac14\left(\frac{f'}{f}\right)^2|\hat{\n}\varphi|^2-\frac{2nv^2}{F^2|\hat{\n}\varphi|^2}\sum_{k\ge 2}F^{kk}\varphi_{kk}^2- \frac{2nv^2}{F^2}\sum_{k\ge 2}F^{kk}
\\&\quad -\frac{f'}{f}\frac{2n}{F}-\left[\left(\frac{f'}{f}\right)^2\frac{|\hat{\n}\varphi|^2}{v^2}-\frac{f''}{f}\right]\frac{n}{F^2}F^{11}|\hat{\n}\varphi|^2
\\&\quad -\frac{nv^2}{F^2} \left(2\la''\la -\la'\frac{f'}{f}\right)F^{i}_i  +\frac{f'}{f}\frac{1}{\la'}+2\frac{\la''\la}{{\la'}^2}.
}
Note that
\eq{\label{C1eq10}
-\frac{f'}{f}\frac{2n}{F}&=-\frac{f'}{f}\frac{2n}{F^2}F_k^{i}(\la'\de_i^k-\tilde{g}^{kj}\varphi_{ij})
\\&=-\frac{f'}{f}\frac{2n\la'}{F^2}F^{i}_i+\left(\frac{f'}{f}\right)^2\frac{n}{F^2}\frac{1}{v^2}F^{11}|\hat{\n}\varphi|^2+ \frac{f'}{f}\frac{2n}{F^2}\sum_{k\ge 2}F^{kk}\varphi_{kk}.
}
Therefore
\eq{\label{C1eq11}
0&\le -\frac{2n}{F^2}F^{11}\frac14\left(\frac{f'}{f}\right)^2|\hat{\n}\varphi|^2-\frac{2nv^2}{F^2|\hat{\n}\varphi|^2}\sum_{k\ge 2}F^{kk}\varphi_{kk}^2- \frac{2nv^2}{F^2}\sum_{k\ge 2}F^{kk}
\\&\quad -\frac{f'}{f}\frac{2n\la'}{F^2}F^{i}_i+\left(\frac{f'}{f}\right)^2\frac{n}{F^2}\frac{1}{v^2}F^{11}|\hat{\n}\varphi|^2+ \frac{f'}{f}\frac{2n}{F^2}\sum_{k\ge 2}F^{kk}\varphi_{kk}
\\&\quad  -\left[\left(\frac{f'}{f}\right)^2\frac{|\hat{\n}\varphi|^2}{v^2}-\frac{f''}{f}\right]\frac{n}{F^2}F^{11}|\hat{\n}\varphi|^2
\\&\quad -\frac{nv^2}{F^2} \left(2\la''\la-\la'\frac{f'}{f}\right)F^{i}_i  +\frac{f'}{f}\frac{1}{\la'}+2\frac{\la''\la}{{\la'}^2}
}
and, completing the square,
\eq{0&\leq  -\frac{n}{F^2}F^{11}\left[\frac12\left(\frac{f'}{f}\right)^2+\left(\frac{f'}{f}\right)^2\frac{|\hat{\n}\varphi|^2}{v^2}-\frac{f''}{f} +2\la''\la { -\la'\frac{f'}{f}} \right]|\hat{\n}\varphi|^2
\\&\quad +\frac{n}{F^2}F^{11}\left[ -2\la'\frac{f'}{f}+\left(\frac{f'}{f}\right)^2\frac{|\hat{\n}\varphi|^2}{v^2}- 2\la''\la+\la'\frac{f'}{f} \right]
\\&\quad -\frac{2n}{F^2}\sum_{k\ge 2}F^{kk}\left(\varphi_{kk}-\frac12\frac{f'}{f}\right)^2-\frac{2n}{F^2|\hat{\n} \varphi|^2}\sum_{k\ge 2}F^{kk}\varphi_{kk}^2
\\&\quad + \frac{2n}{F^2}\left(\frac14\left(\frac{f'}{f}\right)^2-\la' \frac{f'}{f}-v^2\left(1+ \la''\la-\frac12\la'\frac{f'}{f}\right)\right)\sum_{k\ge 2}F^{kk}
\\&\quad +\frac{f'}{f}\frac{1}{\la'}+2\frac{\la''\la}{{\la'}^2}.}

Choose $f(\varphi)=e^{-a \varphi}$ with $a>0$ large enough so that the first term on RHS of \eqref{C1eq11} have negative sign and when $|\hat{\n}\varphi|^2$ is large enough, this term dominates the second term. Also, by choosing $a>0$ large enough and then $|\hat{\n}\varphi|^2$ is large enough, the fourth line and the fifth line are both negative. We get a contradiction. Thus $|\hat{\n}\varphi|^2\le C$.
\end{proof}

\medskip

\section{Preserved convexity in the sphere}\label{sec:Pres-Conv}

In ambient spaces where $\la''$ can be negative it is very difficult to control $F$ from below, if the flow hypersurfaces are not convex. Hence we assume strict convexity in these cases and we have to restrict to space forms to show that this property is preserved.

\begin{prop}\label{pres-conv-sphere}
Let $x_0(M)$ be the embedding of a closed $n$-dimensional manifold $M$ into $\bbS^{n+1}$, such that $x_0(M)$ is strictly convex. 
Let $F$ be a $1$-homogeneous, monotone and inverse concave\footnote{$F(\kappa_1, \cdots, \kappa_n)$ is called inverse concave if $\tilde{F}(\kappa_1, \cdots, \kappa_n)=F^{-1}(\kappa_1^{-1},\cdots, \kappa_n^{-1})$ is concave.} curvature function.
Then along any solution $x$ of \eqref{Flow} with initial embedding $x_0$ all flow hypersurfaces are strictly convex.
\end{prop}

\pf{
Let $b$ be the inverse of the Weingarten map, which exists at least for a short time. We show that for a smooth solution 
\eq{x\cn [0,T^*)\x M\ra \bbS^{n+1}}
all $M_t$, $t<T^*$, are strictly convex. 
From 
\eq{\dot{b}^k_m=-b^k_j\dot{h}^j_ib^i_m,\q F^{qs}b^k_{m;qs}=2F^{qs}b^k_jh^j_{p;q}b^p_rh^r_{i;s}b^i_m-F^{qs}b^k_ph^p_{l;qs}b^l_m,}
\eq{u_{;i}=\la h_{i}^k r_{;k}}
and \eqref{Ev-h-spaceform}
we deduce
\eq{\label{pres-conv-1}\cL b^k_m&=\fr{n}{F^2}\br{\fr{2}{F}F^{rs}F^{pq}-2F^{qs}b^{pr}-F^{pq,rs}}b^k_jb^i_mh_{rs;i}{h_{pq;}}^j\\
			&\hp{=}-\fr{\la^2}{\la'^2}\br{b^k_l{r_{;}}^{l}r_{;m}+{r_{;}}^{k} b_m^lr_{;l}}-\fr{2u\la^2}{\la'^3}b^{k}_{j}{r_{;}}^{j}b^{i}_{m}r_{;i}\\
			&\hp{=}+\psi_1 b^k_m-\fr{2u}{\la'}b^k_lb^l_m+\psi_2\delta_{m}^k,}
where $\psi_i,$ $i=1,2$ are some functions, which are bounded on every compact interval $[0,T_0]\sub [0,T^*)$. If the convexity is lost at some time $T_0<T^*$, then the largest eigenvalue of $b$ blows up at $T_0$. Although the largest eigenvalue is not a smooth function, we can still apply \eqref{pres-conv-1} to estimate it by using the following well known trick, compare e.g. the proof of \cite[Lemma~6.1]{Gerhardt:01/1996}:

Define
\eq{\phi=\sup\{b_{ij}\eta^i\eta^j\cn g_{ij}\eta^i\eta^j=1\}}
and suppose this function attains a maximum at $(t_0,\xi_0)$, $t_0<T_0.$ Using normal coordinates around $(t_0,\xi_0)$ with
\eq{g_{ij}=\de_{ij},\q b_{ij}=\ka_i^{-1}\de_{ij},\q \ka_1^{-1}\leq \dots\leq \ka_n^{-1}.}
Around $(t_0,\xi_0)$ let $\eta$ be the vector field
\eq{\eta=(0,\dots,0,1)} and define
\eq{\ti\phi=\fr{b_{ij}\eta^i\eta^j}{g_{ij}\eta^i\eta^j},}
then locally around $(t_0,\xi_0)$ we have $\ti\phi\leq \phi$ and at this point there holds
\eq{\dot{\ti{\phi}}=\dot{b}_{nn}+2\cF=\dot{b}^n_n}
and the spatial derivatives also coincide. Thus at $(t_0,\xi_0)$ the function $\ti\phi$ and $b^n_n$ satisfy the same evolution equation, whence it suffices to show that the right hand side of \eqref{pres-conv-1} is negative at the point $(t_0,\xi_0)$.

The first line is negative due to the inverse concavity of $F$, compare the proof in \cite[p.~112]{Urbas:/1991}, while for the rest the good terms involving $b^k_lb^l_m$ are surely dominating. This completes the proof.
}

\medskip

\section{Bounds on the speed and the curvature}\label{sec:Speed-Bound}
In this section we deduce the remaining ingredients which are necessary to obtain longtime existence, namely we need a full bound on the second fundamental form and in turn, to apply the Krylov-Safonov theory, we need a lower bound on the curvature function to show that the operator $\cL$ is uniformly parabolic along the flow. We start with the spherical case.

\subsection{The spherical case}

\begin{lemma}\label{H-bound-sphere}
Let $x_0(M)$ be the embedding of a closed $n$-dimensional manifold $M$ into $\bbS^{n+1}$, such that $x_0(M)$ is strictly convex. 
Let 
\eq{F=n\fr{H_k}{H_{k-1}}.}
Then along any solution $x$ of \eqref{Flow} with initial embedding $x_0$ there exists a constant $c=c(n,k,\sup r_0,\inf r_0,\la)$, such that
\eq{\|A\|^2\leq c.}
\end{lemma}

\pf{
Due to the convexity preservation, \cref{pres-conv-sphere}, it suffices to bound the mean curvature $H$ from above. Note $u\ge c_0>0$ by \cref{grad-bound} (we may also use the convexity to get this, cf. \cite[Lemma 2.7.10]{Gerhardt:/2006}). We use the auxiliary function
\eq{w=\log H-\log u} and deduce from \eqref{Ev-u}, \eqref{Ev-h-spaceform}, the concavity of $F$ and
\eq{u_{;i}=\la h^k_i r_{;k},}
that at a maximal point of $w$:
\eq{0\le\cL w&=\fr{1}{H}\cL H-\fr{1}{u}\cL u\\
		&\leq c+\fr{c}{H}-\fr{2n}{FH}\|A\|^2\\
        &\leq c+\fr{c}{H}-\fr{2}{F}H.}
Since $F$ is bounded from above by \cref{F-bound}, we get a upper bound of $H$ from above.
}

We use the previous result to get bounds from below on $F$.

\begin{lemma}\label{1/H-bound-sphere}
Under the assumptions of \cref{H-bound-sphere} there exists a constant $0<c=c(n,\sup r_0,\inf r_0,\la)$ such that
\eq{F\geq c.}
\end{lemma}

\pf{
We use the same method as in \cite[Prop.~5.3]{MakowskiScheuer:11/2016} and bound the auxiliary function
\eq{z=-\log F+f(r),}
where 
\eq{f(r)=-\log\br{\la'-\al},\q 0<\al<\fr{1}{2}\la'(\sup r_0).}
Since $\la''=-\la$, it is direct to check that
\eq{\label{xfff}1-f'\fr{\la'}{\la}=-\fr{\al}{\la'-\al},\q f'^2+f'\fr{\la'}{\la}-f''=0.}
From the convexity, the $1$-homogeneity of $F$ and \cref{H-bound-sphere}, we see
\eq{\label{xff}\fr{n}{F^{2}}F^{ij}h_{ik}h^{k}_{j}\leq \fr{nH}{F}\leq \fr{c}F.}
Using \eqref{Ev-r}, \eqref{EV-F} and \eqref{xff},
\eq{\cL z&=-\fr{1}{F}\cL F-\fr{n}{F^4}F^{ij}F_{;i}F_{;j}+f'\cL r-f''\fr{n}{F^2}F^{ij}r_{;i}r_{;j}\\
		&\leq \fr{n}{F^2}F^{ij}(\log F)_{;i}(\log F)_{;j}+\fr{c}{F}+c+\fr{n}{F^2}F^{ij}g_{ij}\\
        &\hp{=}+f'\fr{c}{F}-f'\fr{n\la'}{\la F^2}F^{ij}g_{ij}+f'\fr{n\la'}{\la F^2}F^{ij}r_{;i}r_{;j}-f''\fr{n}{F^2}F^{ij}r_{;i}r_{;j}.}
At a maximal point of $z$, we use $(\log F)_{;i}=f'r_{;i}$ and
\eqref{xfff}
to obtain
\eq{0\le \cL z&\leq \fr{n}{F^2}F^{ij}r_{;i}r_{;j}\br{f'^2+f'\fr{\la'}{\la}-f''}+\fr{n}{F^2}F^{ij}g_{ij}\br{1-f'\fr{\la'}{\la}}+\fr{c}{F}+c+f'\fr{c}{F}\\
		&=-\fr{\al}{\la'-\al}\fr{n}{F^2}F^{ij}g_{ij}+\fr{c}{F}+c\\
        &<0,}
if $F$ is small enough, since $F^{ij}g_{ij}\ge n$.  
}

Now we finish the a priori estimates in the spherical case.

\begin{prop}\label{LTE-sphere}
Let $x_0(M)$ be the embedding of a closed $n$-dimensional manifold $M$ into $\bbS^{n+1}$, such that $x_0(M)$ is strictly convex. 
Let 
\eq{F=n\fr{H_k}{H_{k-1}}.}
Then any solution $x$ of \eqref{Flow} with initial embedding $x_0$ exists for all positive times with uniform $C^{\8}$-estimates. 
\end{prop}

\pf{
We have uniform $C^2$-bounds from \cref{pres-conv-sphere} and \cref{H-bound-sphere}. Due to \cref{1/H-bound-sphere} we know that the principal curvatures range within a compact subset of the domain on definition of $F$. Hence we have the uniform parabolicity of the operator $\cL$. Due to the concavity of the operator, we can apply the regularity theory of Krylov and Safonov, \cite{Krylov:/1987}, to deduce $C^{2,\al}$ bounds and in turn $C^{\8}$ bounds using the Schauder theory. Thus we can extend the flow beyond any finite $T$.
}

\subsection{The general case}

We provide the bounds on the principal curvatures and on the curvature function from below in case of mild assumptions on the warping factor.

\begin{prop}\label{1/F-bound-general}
Let $a,b\in\R$ and $(N,\bar g)$ be the warped space $((a,b)\x \bbS^n,dr^2+\la^2(r)\s)$ with $\la>0$, $\la'>0$ and $\la''\geq 0$.
Let $F\in C^{\8}(\G)$ be a $1$-homogeneous, strictly monotone and concave curvature function
and let $x_0(M)$ be the embedding of a closed $n$-dimensional manifold $M$ into $N$, such that $x_0(M)$ is a graph over the domain $\bbS^n$ and such that $\ka\in \G$ for all $n$-tupels of principal curvatures along $x_0(M)$. Then along any solution $x$ of \eqref{Flow} with initial embedding $x_0$ there exists a positive constant $c=c(n,\sup r_0,\inf r_0,\la)$, such that
\eq{F\geq c.}
\end{prop}

\begin{rem}Proposition \ref{1/F-bound-general} is the only place where we use $\la''\ge 0$ for proving Theorem \ref{General-flow-main}.
\end{rem}
\pf{
We deduce the evolution of the function $\del_t\p$, where $\p$ is defined as in \eqref{phi}. Recall that there holds \eqref{gamma0},
\eq{
\del_t\varphi&=\frac{1}{\la}\left(\frac{n}{F\left(\frac{\la'}{\la v}\de_i^j-\frac{1}{\la v}\tilde{g}^{jk}\varphi_{ki}\right)}-\frac{u}{\la'}\right)v\\&=\frac{n v^2}{F(\la'\de_i^j- \tilde{g}^{jk}\varphi_{ki})}-\frac{1}{\la'}=:G(\varphi, \hat{\n} \varphi, \hat{\n}^2 \varphi),}
where $\tilde{g}^{ij}=\s^{ij}-\frac{\varphi^i\varphi^j}{v^2}.$
Differentiation gives
\eq{\del_t(\del_t\p)&=G^{ij}(\del_t\p)_{ij}+G^{\p_p}(\del_t\p)_p +G^{\p}\del_t\p.}
From \eqref{G-phi} we obtain 
\eq{G^{\p}\leq-\fr{n^2v^2\la''\la}{F^2}+\fr{\la''\la}{\la'^2}=-\fr{\la''\la}{v^2}\fr{n^2v^4}{F^2}+\fr{\la''\la}{\la'^2}=-\fr{\la''\la}{v^2}\br{\del_{t}\p+\fr{1}{\la'}}^2+\fr{\la''\la}{\la'^2}.}
Since we already have $v\leq c$ due to \cref{grad-bound}, the third order leading term is dominating with a non-positive sign. The maximum principle gives an upper bound for $\del_t\p$ and hence the result.
}

\begin{prop}\label{A-bound-general}
Under the assumptions of \cref{1/F-bound-general} there exists a positive constant $c=c(n,\sup r_0,\inf r_0,\la)$, such that
\eq{\|A\|^{2}\leq c.}
\end{prop}

\pf{In applying the maximum principle to the evolution of $(h^{i}_{j})$ we proceed similarly to the proof of \cref{pres-conv-sphere}. Define
\eq{\phi=\sup\{h_{ij}\eta^i\eta^j\cn g_{ij}\eta^i\eta^j=1\}}
and suppose the function
\eq{w=\log \phi+f(u)+\al r}
 attains a maximum at $(t_0,\xi_0)$, $t_0<T_0,$ where $f$ is defined by
 \eq{f(u)=-\log(u-\be),}
 where $\beta=\frac12\min u$.
 Note that
 \eq{1+f'u=\frac{-\beta}{u-\beta}<0.}
 Using normal coordinates around $(t_0,\xi_0)$ with
\eq{g_{ij}=\de_{ij},\q h_{ij}=\ka_i\de_{ij},\q \ka_1\leq \dots\leq \ka_n,}
and using \eqref{Ev-r}, \eqref{Ev-u} and \eqref{ev-h},
we may pretend that the evolution equation of $w$ at the point $(t_{0},\xi_{0})$ is given by
\eq{\label{A-bound-general-1}\cL w&\leq  \fr{n}{F^{2}}\fr{2}{\ka_{n}-\ka_{1}}\sum_{k=1}^{n}(F^{nn}-F^{kk})(h_{nk;n})^{2}(h^{n}_{n})^{-1}+c+\fr{c}{\ka_{n}}+\fr{n}{F^{2}}F^{kl}h_{rk}h^{r}_{l}\\
				&\hp{=}-\fr{2n}{F}\ka_{n}+\fr{c(1+\ka_n^{-1})}{F^2}F^{ij}g_{ij}+\fr{n}{F^{2}}F^{{ij}}(\log h^{n}_{n})_{;i}(\log h^{n}_{n})_{;j}\\
					&\hp{=}+\fr{n}{F^{2}}\br{F^{kl}h_{rk}h^{r}_{l}-\fr 1n F^{2}}f'u-f'\fr{\la''\la}{\la'^{2}}\|\n r\|^{2}u+c|f'|\\
					&\hp{=}-f''\fr{n}{F^{2}}F^{ij}u_{;i}u_{;j}+\fr{\al c}{F}-\al\fr{\la}{\la'}-\fr{n\al\la'}{\la F^{2}}F^{ij}(g_{ij}-r_{;i}r_{;j}),}
where we used a trick that already appeared in the proof of \cite[Prop.~6.3]{Enz:10/2008} and in a similar fashion in \cite[Thm.~9.7]{Gerhardt:/2003}, namely that due to the concavity of $F$ there holds
\eq{F^{kl,rs}\eta_{kl}\eta_{rs}\leq \sum_{k\neq l}\fr{F^{kk}-F^{ll}}{\ka_{k}-\ka_{l}}\eta_{kl}^{2}\leq \fr{2}{\ka_{n}-\ka_{1}}\sum_{k=1}^{n}(F^{nn}-F^{kk})\eta_{nk}^{2}}
for all symmetric matrices $(\eta_{kl}),$
cf. \cite[Lemma~2.1.14]{Gerhardt:/2006}. 
Furthermore we have 
\eq{F^{nn}\leq \dots \leq F^{11},}
cf. \cite[Lemma~2]{EckerHuisken:02/1989}.
In order to estimate \eqref{A-bound-general-1}, we distinguish two cases.

{\bf{Case 1:}} $\ka_{1}<-\e_{1}\ka_{n}$, $0<\e_{1}<\fr 12$. Then
\eq{F^{ij}h_{ik}h^{k}_{j}\geq F^{11}\ka_{1}^{2}\geq \fr{1}{n}F^{ij}g_{ij}\ka_{1}^{2}\geq \fr 1n F^{ij}g_{ij}\e_{1}^{2}\ka_{n}^{2}.}
We use $\n w=0$ to estimate
\eq{\fr{n}{F^{2}}F^{ij}(\log h^{n}_{n})_{;i}(\log h^{n}_{n})_{;j}=f'^{2}\fr{n}{F^{2}}F^{ij}u_{;i}u_{;j}+f'\fr{2n\al}{F^{2}}F^{ij}u_{;i}r_{;j}+\fr{n\al^{2}}{F^{2}}F^{ij}r_{;i}r_{;j}.}
If $\ka_{n}$ is sufficiently large, in this case \eqref{A-bound-general-1} becomes
\eq{\cL w&\leq \fr{1}{F^{2}}F^{ij}g_{ij}(\e_{1}^{2}\ka_{n}^{2}(1+f'u)+(c+|f'|\al)\ka_{n}+c\al^{2}+c)+c(|f'|+1)\\
	&\hp{=}-\fr{2n}{F}(\ka_n-\al c)-\al\fr{\la}{\la'}-\fr{n}{F^{2}}F^{ij}u_{;i}u_{;j}(f''-f'^{2}),}
which is negative for large $\ka_{n}$, after fixing $\al_0=\al_0(M_0,\sup r_0,\inf r_0,\la)$ large enough to ensure
\eq{c(|f'|+1)-\al_0\fr{\la}{\la'}<0.}
We also use $1+f'u\leq c<0$ and $f''-f'^{2}=0$. Hence in this case any $\al\geq \al_0$ yields an upper bound for $\ka_{n}$.

{\bf{Case 2:}} $\ka_{1}\geq -\e_{1}\ka_{n}.$ Then
\eq{&\fr{2}{\ka_{n}-\ka_{1}}\sum_{k=1}^{n}(F^{nn}-F^{kk})(h_{nk;n})^{2}(h^{n}_{n})^{-1}\\
	\leq~& \fr{2}{1+\e_{1}}\sum_{k=1}^{n}(F^{nn}-F^{kk})(h_{nk;n})^{2}(h^{n}_{n})^{-2}\\
	\leq~&\fr{2}{1+\e_{1}}\sum_{k=1}^{n}(F^{nn}-F^{kk})(h_{nn;k})^{2}(h^{n}_{n})^{-2}+c(\e_{1})\sum_{k=1}^{n}(F^{kk}-F^{nn})\ka_{n}^{-2}\\
	&\hp{=}+\fr{4}{1+\e_{1}}\sum_{k=1}^{n}(F^{nn}-F^{kk})h_{nn;k}\bar{R}_{\al\be\g\de}\nu^{a}x^{\be}_{\ ;n}x^{\g}_{\ ;n}x^{\de}_{\ ;k}(h^{n}_{n})^{-2}\\
	\leq~& \fr{2}{1+2\e_{1}}\sum_{k=1}^{n}(F^{nn}-F^{kk})(h_{nn;k})^{2}(h^{n}_{n})^{-2}+c(\e_{1})\sum_{k=1}^{n}(F^{kk}-F^{nn})\ka_{n}^{-2}, }
where we used the Codazzi equation \eqref{Codazzi} and the Cauchy-Schwarz inequality.
We deduce further:
\eq{&F^{ij}(\log h^{n}_{n})_{;i}(\log h^{n}_{n})_{;j}+\fr{2}{\ka_{n}-\ka_{1}}\sum_{k=1}^{n}(F^{nn}-F^{kk})(h_{nk;n})^{2}(h^{n}_{n})^{-1}\\
	\leq~&\fr{2}{1+2\e_{1}}\sum_{k=1}^{n}F^{nn}(\log h^{n}_{n})_{;k}^{2}-\fr{1-2\e_{1}}{1+2\e_{1}}\sum_{k=1}^{n}F^{kk}(\log h^{n}_{n})_{;k}^{2}+c(\e_{1})F^{ij}g_{ij}\ka_{n}^{-2}\\
	\leq~&\sum_{k=1}^{n}F^{nn}(\log h^{n}_{n})_{;k}^{2}+c(\e_{1})F^{ij}g_{ij}\ka_{n}^{-2}\\
	=~&c(\e_1)F^{ij}g_{ij}\ka_n^{-2}+f'^{2}F^{nn}\|\n u\|^{2}+2\al f'F^{nn}\ip{\n u}{\n r}+\al^{2}F^{nn}\|\n r\|^{2}.}
We plug this into \eqref{A-bound-general-1} and obtain for large $\ka_{n}$:
\eq{\cL w&\leq c+\fr{n}{F^2}F^{nn}\ka_n^2(1+f'u)-\fr{2n}{F}(\ka_n-\al c)+\fr{1}{F^2}F^{ij}g_{ij}\br{c+c(\e_1)-\fr{n\al\la'}{v^2 \la}}\\
		&\hp{=}-f''\fr{n}{F^2}F^{ij}u_{;i}u_{;j}-\al\fr{\la}{\la'}+f'^2\fr{n}{F^2}F^{nn}\|\n u\|^2+\fr{2n\al f'}{F^{2}}F^{nn}\ip{\n u}{\n r}\\
        &\hp{=}+\fr{n\al^{2}}{F^{2}}F^{nn}\|\n r\|^{2}\\
		&\leq \fr{n}{F^{2}}F^{nn}\br{\ka_{n}^{2}(1+f'u)+2\al |f'|c\ka_{n}+\al^{2}\|\n r\|^{2}}-\fr{2n}{F}(\ka_{n}-\al c)\\
        &\hp{=}+c-\al\fr{\la}{\la'}+\fr{1}{F^2}F^{ij}g_{ij}\br{c+c(\e_1)-\fr{n\al\la'}{v^2 \la}}\\
		&<0}
after possibly enlarging $\al$ even further (compared to case 1) and for large $\ka_{n}$. This completes the proof.
}

As in \cref{LTE-sphere} we conclude:

\begin{prop}\label{LTE-general}
Under the assumptions of \cref{A-bound-general} the flow \eqref{Flow} exists for all times with uniform $C^{\8}$-estimates.
\end{prop}

\medskip

\section{Proofs of the main theorems}\label{sec:Pf-Main}

We give the final arguments to complete the proofs concerning the flow results and start with the spherical case.

\begin{proof}[Proof of \cref{Sphereflow-main}]\label{Pf-Sphereflow-main}

In order to complete the proof of \cref{Sphereflow-main} with the help of \cref{LTE-sphere}, all we have to show is that each subsequential limit is a sphere independent of the subsequence as $t\ra \8$.

The evolution of the weighted enclosed volume
\eq{V(t)=\int_{\Om_{t}}\la' dN} is
\eq{\dot{V}(t)=\int_{M_{t}}\br{\fr{n\la'}{F}-u}d\mu_t\geq \int_{M_{t}}\br{\fr{n\la'}{H}-u}d\mu_t\ge 0.}
The first inequality is due to the concavity of $F$ which implies $F\leq H,$ \cite[Lemma 2.2.20]{Gerhardt:/2006} and the second one is
due to Brendle's Heintze-Karcher type inequality, \cite[equ.~(4)]{Brendle:06/2013}. That is, $V$ is increasing. Since $V$ is obviously bounded we have
\eq{\int_{0}^{\8}\int_{M_{t}}\br{\fr{n\la'}{H}-u}~d\mu_{t}dt<\8}
and hence
\eq{\int_{M_{t}}\br{\fr{n\la'}{H}-u}~d\mu_{t}\ra 0.}
So any convergent subsequence of $M_{t}$ must converge to a sphere, due to the characterization of the limiting case in the Heintze-Karcher inequality. Due to the spherical barriers this sphere is unique and we conclude the proof of the theorem.
\end{proof}

Now we turn to the other case and prove \cref{General-flow-main}.

\begin{proof}[Proof of \cref{General-flow-main}]\label{Pf-General-flow-main}
 Again it suffices to prove that there exists a subsequence that converges to a sphere. If no subsequence converges to a geodesic sphere then there can not be any subsequence for which $\|\n r\|\ra 0$. Hence there exists a positive constant $c$ such that for all times $t>0$ we have
 \eq{\label{General-pf-1}\max_{M_{t}}\|\n r\|^{2}\geq c.}
 The area evolves according to
 \eq{\label{Mon-area}\fr{d}{dt}|M_t|=\int_{M_{t}}\cF H\geq \int_{M_{t}}\br{n-\fr{Hu}{\la'}}=\int_{M_{t}}\fr{{\rm div}(\la\n r)}{\la'}=\int_{M_{t}}\fr{\la''\la}{\la'^{2}}\|\n r\|^{2}\ge 0.}
 The inequality in \eqref{Mon-area} is again due to $F\le H$. The last two equalities in \eqref{Mon-area} follow from the fact ${\rm div}(\la\n r)= n\la'-Hu$ and integration by parts respectively.
 
Due to the $C^{1}$-estimates the area is bounded and hence, because of $\la''\geq 0$, every subsequential limit $M_{t}\ra \ti M$ must satisfy
\eq{\int_{\ti M}\fr{\la''\la}{\la'^{2}}\|\n r\|^{2}=0,}
whence
\eq{\label{General-pf-2}\la''\|\n r\|^{2}=0}
throughout any subsequential limit. For all $t>0$, let
\eq{\xi_{t}:=\argmax_{M_{t}} \|\n r\|^{2}.}
We obtain that
\eq{\la''(\xi_{t})\ra 0,\q t\ra 0,}
for otherwise we reach a contradiction to \eqref{General-pf-1} and \eqref{General-pf-2}.
From \eqref{C1eq3} we obtain at the points $(t,\xi_{t})$,
\eq{\cL|\hat \n\p|^{2}&\leq -2G^{ij}\s_{ij}|\hat\n\p|^{2}+2G^{ij}\p_{i}\p_{j}+2G^{\p}|\hat\n\p|^{2}\\
			&\leq-\fr{2nv^{2}}{F^{2}}F^{i}_{k}\ti{g}^{kj}\s_{ij}|\hat\n\p|^{2}+\fr{2nv^{2}}{F^{2}}F^{i}_{k}\ti{g}^{kj}\p_{i}\p_{j}+c\la''|\hat\n\p|^{2}\\
			&\leq -\e|\hat\n\p|^{2},}
for some suitable $\e>0$. Thus $|\hat\n\p|^{2}$ actually has to decay exponentially and we obtain a contradiction to \eqref{General-pf-1}. 
 \end{proof}
 
 \medskip
 
\section{Geometric inequalities}\label{sec:Geom-Ineq}

In this section we complete the proof of the geometric inequalities. First of all, along the flow $\frac{d}{dt}x=\mathcal{F}\nu,$ we have the following variational formulas.
\begin{prop}\label{evovl-integ-quant}
Let $M_t\subset N$ be a family of closed hypersurfaces evolving by $\frac{d}{dt}x=\mathcal{F}\nu$. Denote by $\Om_t$ the enclosed domain by $M_{t}$ and $\{a\}\x \bbS^{n}$. Then
\begin{eqnarray}\label{e0}
\frac{d}{dt}\int_{\Omega_t} f=\int_{M_t}f \mathcal{F} \q\fa f\in C^{\infty}(M),
\end{eqnarray}
and
\begin{eqnarray}\label{e1}
\frac{d}{dt}|M_t|=\int_{M_t} H\cF.
\end{eqnarray}
If $\bar{\Delta}\la'\bar g-\bar{\nabla}^2\la' +\la'\ov\Rc =0$, then
\begin{eqnarray}\label{e2}
\frac{d}{dt}\int_{M_t} H\la'= \int_{M_t} (2\sigma_2 \la'+2H\left<\bar{\nabla} f,\nu\right>)\cF.
\end{eqnarray}
\end{prop}
\begin{proof}The first and second ones are well known and have already been used in \cref{Pf-Sphereflow-main}. We compute the third one.
\eq{
\frac{d}{dt}\int_{M_t} H\la'&= \int_{M_t} \la'(-\Delta \cF-\cF|A|^2-\cF \ov\Rc(\nu,\nu))\\
	&\hp{=}+\int_{M_t} \br{H\left<\bar{\nabla} \la',\nu\right>+H^2\la'}\cF\\
    &=\int_{M_t} -(\bar{\Delta}\la'-\bar{\nabla}^2\la'(\nu,\nu)-H\left<\bar{\nabla}\la',\nu\right>+\la'\ov\Rc(\nu,\nu)) \cF\\
    &\hp{=}+\int_{M_t} \br{H\left<\bar{\nabla} \la',\nu\right>+(H^2-|A|^2)\la'}\cF\\
    &=\int_{M_t} 2\sigma_2 \la'\cF+2H\left<\bar{\nabla} f,\nu\right>\cF.
}
\end{proof}

\begin{prop}\label{Mink-ineq-2}
Let $\Sigma\subset  N$ be a closed hypersurface. 
 If $\Sigma$ is star-shaped and $\frac{\la''}{\la}+\frac{1-\la'^2}{\la^2}\geq 0$, then
\eq{\label{Min2}
\int_\Sigma (n-1)H\la'\leq \int_\Sigma 2\sigma_2u.
}
\end{prop}

\begin{proof}


Multiplying $\s_2^{ij}$ to \eqref{graph-h}, summing over $i,j$, integrating over $\Sigma$ and using
\eq{\n_i\s_2^{ij}={h^i_{i;m}}g^{mj}-{h^{ij}}_{;i}=-\ov\Rc(\nu,x_{;m})g^{mj},}
we have
\eq{
\int_\Sigma (n-1)H\la'-2\sigma_2u&=\int_\Sigma \s_2^{ij}(\la r_{;j})_{;i}\\
    &=\int_\Sigma \la\ov\Rc(\nu, x_{;m}){r_{;}}^m\\
    &=\int_\Sigma -(n-1)\left[\frac{\la''}{\la}+\frac{1-\la'^2}{\la^2}\right]\la\|\nabla r\|^2\left<\del_r,\nu\right>\\
    &\leq 0,
}
where we used the starshapedness and \eqref{Ricci-WP}.
\end{proof}

Now we choose the flow as
\begin{eqnarray}\label{flow2}
\frac{d}{dt}x=\left(\frac{n}{H}-\frac{u}{\la'}\right)\nu.
\end{eqnarray}
Along this flow the area $|M_t|$ is non-decreasing and the quantity
\eq{\int_{M_t} H\la'd\mu_t-2n\int_{\Omega_t}\frac{\la'\la''}{\la} dN}
is non-increasing.

\begin{prop}\label{monot} Under the assumptions of \cref{Geom-ineq},
let $M_t\subset N$ be a family of closed star-shaped hypersurfaces evolving by \eqref{flow2}. Then
\begin{eqnarray}\label{WeigV}
\frac{d}{dt}\int_{\Omega_t} \la' dN\geq 0, 
\end{eqnarray}
\begin{eqnarray}\label{Area}
\frac{d}{dt}|M_t|\geq 0
\end{eqnarray}
and
\begin{eqnarray}\label{HV}
\frac{d}{dt}\left(\int_{M_t} H\la'd\mu_t-2n\int_{\Omega_t}\frac{\la'\la''}{\la}dN \right)\leq 0.
\end{eqnarray}
\end{prop}
\begin{proof}
We first note that all the assumptions in \cref{evovl-integ-quant} and \cref{Mink-ineq-2} are satisfied by the anti-de-Sitter Schwarzschild space and the hyperbolic space. Also  the Heintze-Karcher type inequality holds for the anti-de-Sitter Schwarzschild space and the hyperbolic space. Thus inequality \eqref{WeigV} is proved in the same way as  the proof of \cref{Sphereflow-main} in \cref{Pf-Sphereflow-main}. Inequality \eqref{Area} was proved in the proof of \cref{General-flow-main} in \cref{Pf-Sphereflow-main}. Next we show \eqref{HV}. From \eqref{e2} and \eqref{e0}, we have
\eq{\label{monot-1}
&\frac{d}{dt}\left(\int_{M_t} H\la'-2n\int_{\Omega_t}\frac{\la'\la''}{\la} \right)\\
=~&\int_{M_t} \left(2\sigma_2 \la'+2H\left<\bar{\nabla}\la',\nu\right>-2n\frac{\la'\la''}{\la}  \right)\left(\frac{n}{H}-\frac{u}{\la'}\right)\\
=~&\int_{M_t} 2\sigma_2 \la'\left(\frac{n}{H}-\frac{u}{\la'}\right)+\int_{M_t} \frac{\la''}{\la} \left(2Hu-2n\la'\right)\left(\frac{n}{H}-\frac{u}{\la'}\right)\\
\leq~& \int_{M_t} \left((n-1)H \la'-2\sigma_2u\right)-\int_{M_t} 2H\frac{\la'\la''}{\la}\left(\frac{n}{H}-\frac{u}{\la'}\right)^2
\\
\leq~& 0.}
In the second equality, we used $\left<\bar{\nabla} \la',\nu\right>=\frac{\la''}{\la}u$ and in the last two inequalities we used Newton-Maclaurin inequality, \eqref{Min2} and $\la''\ge 0$.
\end{proof}

The inequalities in \cref{Geom-ineq} follows immediately from the monotonicity in Proposition \ref{monot} and the convergence result of the flow. The classification of the equality case follows easily by checking the equality in \eqref{monot-1}. \qed

\medskip

\noindent{\bf Acknowledgements.} Part of this work was done while CX was visiting the mathematical institute of Albert-Ludwigs-Universit\"at Freiburg. He would like to thank the institute for its hospitality. The authors are grateful to Professor Pengfei Guan for his comments and for sending us their private note. The authors would also like to thank Professor Guofang Wang for bringing the paper \cite{GeWangWu:10/2015} to their attention. Research of CX is  supported in part by NSFC (Grant No.~11501480) and  the Natural Science Foundation of Fujian Province of China (Grant No.~2017J06003). 

\medskip

\bibliographystyle{amsplain}
\bibliography{Bibliography.bib}

\end{document}

%% file: StandardPaper3.tex
\usepackage[colorlinks=true,linkcolor=blue,citecolor=blue,urlcolor=blue,hypertexnames=false]{hyperref}
\usepackage{bookmark}
\usepackage{amsthm,thmtools,amssymb,amsmath,amscd}

\usepackage{fancyhdr}
\usepackage{esint}
\usepackage{enumerate}

\usepackage{pictexwd,dcpic}

\swapnumbers
\declaretheorem[name=Theorem,numberwithin=section]{thm}
\declaretheorem[name=Remark,style=remark,sibling=thm]{rem}
\declaretheorem[name=Lemma,sibling=thm]{lemma}
\declaretheorem[name=Proposition,sibling=thm]{prop}

\declaretheorem[name=Corollary,sibling=thm]{cor}
\declaretheorem[name=Assumption,style=definition,sibling=thm]{assum}

\numberwithin{equation}{section}

\usepackage{cleveref}
\crefname{lemma}{Lemma}{Lemmata}
\crefname{prop}{Proposition}{Propositions}
\crefname{thm}{Theorem}{Theorems}
\crefname{cor}{Corollary}{Corollaries}
\crefname{defn}{Definition}{Definitions}
\crefname{example}{Example}{Examples}
\crefname{rem}{Remark}{Remarks}
\crefname{assum}{Assumption}{Assumptions}
\crefname{notation}{Notation}{Notation}

\usepackage{autonum}

\newcommand{\ti}{\tilde}

\newcommand{\cn}{\colon}
\newcommand{\sub}{\subset}
\newcommand{\ov}{\overline}

\newcommand{\R}{\mathbb{R}}

\newcommand{\bbS}{\mathbb{S}}

\newcommand{\8}{\infty}

\newcommand{\al}{\alpha}
\newcommand{\be}{\beta}
\newcommand{\g}{\gamma}
\newcommand{\de}{\delta}
\newcommand{\e}{\epsilon}
\newcommand{\ka}{\kappa}
\newcommand{\la}{\lambda}

\newcommand{\s}{\sigma}
\newcommand{\Si}{\Sigma}
\newcommand{\p}{\varphi}

\newcommand{\Om}{\Omega}
\newcommand{\D}{\Delta}
\newcommand{\G}{\Gamma}

\newcommand{\cL}{\mathcal{L}}

\newcommand{\cF}{\mathcal{F}}

\newcommand{\del}{\partial}
\newcommand{\n}{\nabla}

\newcommand{\fa}{\forall}
\newcommand{\rt}{\sqrt}

\newcommand{\ip}[2]{\left\langle #1,#2 \right\rangle}
\newcommand{\fr}[2]{\frac{#1}{#2}}
\newcommand{\x}{\times}

\DeclareMathOperator{\Rm}{Rm}
\DeclareMathOperator{\Rc}{Rc}

\DeclareMathOperator{\grad}{grad}

\DeclareMathOperator{\argmax}{argmax}

\newcommand{\pf}[1]{\begin{proof} #1 \end{proof}}
\newcommand{\eq}[1]{\begin{equation}\begin{alignedat}{2} #1 \end{alignedat}\end{equation}}

\newcommand{\br}[1]{\left(#1\right)}

\newcommand{\ra}{\rightarrow}
\newcommand{\hra}{\hookrightarrow}

\newcommand{\mrm}{\mathrm}

\newcommand{\hp}{\hphantom}
\newcommand{\q}{\quad}

